\documentclass{amsart}
\usepackage{latexsym}
\usepackage{graphicx,subfigure}
\usepackage{array}
\usepackage{indentfirst}
\usepackage{fancyhdr}
\usepackage{epsf}
\usepackage{amsmath,bm}
\usepackage{amsthm}
\usepackage{booktabs}
\usepackage{amsfonts}
\usepackage{epsf}
\usepackage{multicol}
\usepackage[square, comma, sort&compress, numbers]{natbib}
\usepackage{fancyhdr}
\usepackage{pxfonts}
\usepackage{caption}
\usepackage{epstopdf}
\usepackage{algorithm}
\usepackage{algpseudocode}
\usepackage{exscale}
\usepackage{relsize}
\usepackage{setspace}
\usepackage{geometry}
\usepackage{tikz}
\usetikzlibrary{arrows,shapes,chains,positioning}
\usetikzlibrary{decorations.markings}
\tikzstyle arrowstyle=[scale=1]
\tikzstyle directed=[postaction={decorate,decoration={markings,mark=at position .65 with {\arrow[arrowstyle]{stealth}}}}]
\tikzstyle reverse directed=[postaction={decorate,decoration={markings,mark=at position .65 with {\arrowreversed[arrowstyle]{stealth};}}}]

\newtheorem{Def}{Definition}
\newtheorem{Th}{Theorem}[section]
\newtheorem{Lm}{Lemma}[section]

\newtheorem{remark}{Remark}[section]
\newcommand{\bc}{\begin{center}}
\newcommand{\ec}{\end{center}}
\newcommand{\be}{\begin{eqnarray}}
\newcommand{\ee}{\end{eqnarray}}

\newcommand{\ben}{\begin{eqnarray*}}
\newcommand{\een}{\end{eqnarray*}}

\newcommand{\Om}{{\rm\Omega}}

\newcommand{\dx}{\,dx}

\newcommand{\ds}{\,ds}

\newcommand{\cE}{\mathcal{E}}

\newcommand{\cT}{\mathcal{T}}
\newcommand{\R}{\mathbb{R}}

\newcommand{\cP}{\ensuremath{\mathcal{P}} }

\renewcommand{\arraystretch}{1.5}

\newcommand{\yemeifont}{\fontsize{9pt}{\baselineskip}\selectfont}
\begin{document}
\title{
High accuracy methods for eigenvalues of elliptic operators by nonconforming elements
}

\author {Jun Hu}
\address{LMAM and School of Mathematical Sciences, Peking University,
  Beijing 100871, P. R. China.  hujun@math.pku.edu.cn}

\author{Limin Ma}
\address{LMAM and School of Mathematical Sciences, Peking University,
  Beijing 100871, P. R. China. maliminpku@gmail.com}
\thanks{The authors were supported by  NSFC
projects 11625101, 91430213 and 11421101}

\maketitle

\begin{abstract}
In this paper, three high-accuracy methods for eigenvalues of second order elliptic operators are proposed by using the nonconforming Crouzeix-Raviart(CR for short hereinafter) element and the nonconforming enriched Crouzeix-Raviart(ECR for short hereinafter) element. They are based on a crucial full one order superconvergence of the first order mixed Raviart-Thomas(RT for short hereinafter) element. The main ingredient of such a superconvergence analysis is to employ a discrete Helmholtz decomposition of the difference between the canonical interpolation and the finite element solution of the RT element. In particular, it allows for some vital cancellation between terms in one key sum of boundary terms. Consequently, a full one order superconvergence follows from a special relation between the CR element and the RT element, and the equivalence between the ECR element and the RT element for these two nonconforming elements. These superconvergence results improve those in literature from a half order to a full one order for the RT element, the CR element and the ECR element. Based on the aforementioned superconvergence of the RT element,  asymptotic expansions of eigenvalues are established and employed to achieve  high accuracy extrapolation methods for these two nonconforming elements. In contrast to a classic analysis in literature, the novelty herein is to use not only the canonical interpolations of these nonconforming elements but also that of the RT element to analyze such asymptotic expansions. Based on the superconvergence of these nonconforming elements,  asymptotically exact a posteriori error estimators of eigenvalues are constructed and analyzed for them. Finally, two post-processing methods are proposed to improve accuracy of approximate eigenvalues by employing these a posteriori error estimators. Numerical tests are provided to justify and compare the performance of the aforementioned methods.

  \vskip 15pt


 \vskip 15pt


\end{abstract}

\section{Introduction}
Eigenvalue problems are important. They appear in many fields, such as quantum mechanics, fluid mechanics, stochastic process, etc. A fundamental work is to find eigenvalues of partial differential equations. The topic about how to approximate eigenvalues with high accuracy attracts more and more interest.

The superconvergence analysis plays an important role in approximating eigenvalues with high accuracy. As is known, there are many results in literature for low order conforming finite elements and mixed elements of second order elliptic problems, see \cite{Chen1995High,Chen2013Superconvergence,Brandts1994Superconvergence,Jan2000Superconvergence,Douglas1989Superconvergence}. However, for nonconforming elements, the reduced continuity of trial and test functions makes the corresponding superconvergence analysis very difficult. So far, most of superconvergence results for nonconforming elements are focused on methods on rectangular or nearly parallelogram triangulations, see \cite{ZHONG2005CONSTRAINED,LIN2005ON,MING2006SUPERCONVERGENCE}. There are a few superconvergence results for nonconforming elements on triangular meshes \cite{Hu2016Superconvergence,li2017global,mao2009high}. In \cite{Hu2016Superconvergence}, a  half order superconvergence was analyzed for the CR element. The main idea therein is to employ a special relation between the CR element and the RT element to explore some conformity of discrete stresses by this nonconforming element. However, a full one order supconvergence was observed in the numerical tests \cite{Hu2016Superconvergence}. The cause of such a gap is from a half order superconvergence for the RT element \cite{Brandts1994Superconvergence}, which is a half lower than the optimal superconvergence indicated by numerical tests. In fact, the analysis in \cite{Brandts1994Superconvergence} for one key sum of boundary terms was dependent on a result of Sobolev spaces, which can not be improved as showed by a counter example in \cite{J1972Non}. Thus, a direct application of this result can only yield a half order superconvergence for the RT element. In \cite{li2017global}, a full one order superconvergence was proved by following the analysis \cite{bank2003asymptotically} for the RT element. The result therein requires the regularity of the primary solution in $H^{4+\epsilon}(\Om,\mathbb{R})$ for any $\epsilon >0$.

In this paper, a new analysis for the aforementioned boundary terms is presented, which leads to a full one order superconvergence for the RT element. The main ingredient of such a superconvergence analysis is to employ a discrete Helmholtz decomposition of the difference between the canonical interpolation and the finite element solution of the RT element. In particular, it allows for some vital cancellation between the boundary terms sharing a common vertex in one key sum. Thus, following the analysis in \cite{Hu2016Superconvergence}, the superconvergence results for the CR element and the ECR element of the Poisson problem are improved from a half order to a full one order. These results are also extended to corresponding eigenvalue problems.

Extrapolation methods are widely used to improve the accuracy of eigenvalues. The mathematical analysis is based on asymptotic expansions of approximate eigenvalues. For the conforming linear element of second order elliptic eigenvalue problems, the corresponding asymptotic expansions were analyzed in \cite{Lin1984Asymptotic}. The extensions for second order elliptic operators to variable coefficient elliptic eigenvalue problems, eigenvalue problems on 3-dimensional domains, eigenvalue problems on domains with reentrant corners, nonconforming elements and mixed finite elements, can be found in \cite{Ding1990quadrature,Lin2011Extrapolation,Blum1990Finite,Lin2008New,Lin2009Asymptotic,lin1999high}, respectively; the extensions for forth order elliptic eigenvalue problems to the Ciarlet-Raviart mixed scheme and nonconforming elements on rectangular meshes were discussed in \cite{Chen2007Asymptotic,Jia2010Approximation,Luo2002High,Lin2007Finite,Lin2009New}, respectively. For the CR element, it was pointed out by Lin in \cite{Lin2005CAN} that the accuracy of discrete eigenvalues can be improved from second order to forth order by extrapolation methods, if corresponding eigenfunctions are smooth enough. But the crucial asymptotic expansions were not established there. The main difficulty is that the canonical interpolation of the CR element does not admit the usual superclose property with respect to the finite element solution in the energy norm.


In this paper, asymptotic expansions of eigenvalues are established and employed to achieve high accuracy extrapolation methods for both the CR element and the ECR element. In contrast to a classic analysis in literature, the novelty herein is to use not only the canonical interpolations of these two nonconforming elements but also that of the RT element to analyze such asymptotic expansions. Then, thanks to the superconvergence property of the RT element, the desired asymptotic expansions follow from a special relation between the CR element and the RT element \cite{Marini1985An}, and the equivalence between the ECR element and the RT element \cite{Hu2015The}, respectively.

Besides the aforementioned extrapolation methods, gradient recovery techniques can also be used to improve the accuracy of eigenvalue approximations. In \cite{Zhang2006Enhancing}, for eigenvalues of the Laplacian operator by the conforming linear element, remarkable fourth order convergence rates of approximate eigenvalues were observed. The enhancements to eigenvalues by gradient recovery techniques are based on a simple identity
$$
\lambda_h-\lambda=\parallel \nabla_h (u-u_h)\parallel_{0,\Om}^2-\lambda \parallel u-u_h\parallel_{0,\Om}^2,
$$
where $(\lambda_h, u_h)$ is the corresponding approximation to an eigenpair $(\lambda, u)$. The first term on the right side of the above identity can be approximated with high accuracy by gradient recovery techniques, such as polynomial preserving recovery techniques(PPR for short hereinafter) in \cite{Zhang2005A}, Zienkiewicz-Zhu superconvergence patch recovery techniques(SPR for short hereinafter) in \cite{Zienkiewicz1992The} and the superconvergent cluster recovery method in \cite{huang2010superconvergent}. Since the second term is of higher order, new approximate eigenvalues with higher accuracy can be obtained by the gradient recovery techniques.

As for nonconforming elements of second order elliptic eigenvalue problems, the identity becomes
\begin{equation*}
\begin{split}
\lambda -\lambda_h =\big(\nabla_h (u -u_h ),\nabla_h (u -u_h )\big)+2 (\nabla u ,\nabla_h u_h )-2\lambda_h (u ,u_h )-\lambda_h (u -u_h ,u -u_h ).
\end{split}
\end{equation*}
Note that the consistency error $(\nabla u ,\nabla_h u_h )-2\lambda_h (u ,u_h )$ relates to eigenfunctions themselves.

In this paper, asymptotically exact a posteriori  error estimators of eigenvalues are constructed and analyzed for both the CR element and the ECR element. As a generalization of the result in \cite{Zhang2006Enhancing}, enhancements to eigenvalues are resulted from the corresponding asymptotically exact a posteriori error estimators. In order to approximate the extra term $(\nabla u,\nabla_h u_h)-2\lambda_h(u,u_h)$ with high accuracy, the canonical interpolations $w_h$ of eigenfunctions are introduced. Thanks to the commuting property of the canonical interpolations of these two nonconforming elements,
$$
(\nabla u,\nabla_h u_h)-\lambda_h(u,u_h)=-\lambda_h(u-w_h,u_h).
$$
%
Thus, by expressing the interpolation error in terms of the derivatives of eigenfunctions, the remaining term $\lambda_h(u-w_h,u_h)$ can be approximated with high accuracy by the gradient recovery techniques. In this way, asymptotically exact a posteriori error estimators of eigenvalues are accomplished. Furthermore, a summation of approximate eigenvalues and the corresponding asymptotically exact a posteriori error estimators produces new approximate eigenvalues with higher accuracy.

An additional technique to improve the accuracy of eigenvalue approximations is to combine two approximate eigenvalues by a weighted-average \cite{Hu2012A}. Although an eigenvalue is indeed a weighted-average of two approximations, the corresponding weights are usually unknown because they relate to the errors of the two discrete eigenvalues. The main idea in \cite{Hu2012A} is to design approximate weights through four approximate eigenvalues, which needs two methods to compute two upper bounds and two lower bounds of eigenvalues on two meshes, respectively.


Thanks to the aforementioned asymptotically exact a posteriori error estimators, a new combining technique is proposed to obtain approximate eigenvalues with high accuracy. Given  lower bounds of eigenvalues and the corresponding nonconforming approximate eigenfunctions, conforming approximations of eigenfunctions are obtained by applying the average-projection method \cite{Hu2015Constructing} to these nonconforming eigenfunctions. Asymptotical upper bounds of eigenvalues can be obtained by taking the  Rayleigh quotients of such  conforming functions, see \cite{Hu2015Constructing} for more details. Finally, for the lower and the upper bounds of eigenvalues, the weights are designed by using the corresponding asymptotically exact a posteriori error estimators. Furthermore, the superconvergence of the resulted combining eigenvalues is proved. It needs to point out that our algorithm only needs to solve a discrete eigenvalue problem on one mesh, while the one in \cite{Hu2012A} needs to solve four discrete eigenvalue problems on two meshes.

The remaining paper is organized as follows. Section 2 presents second order elliptic eigenvalue problems and some notations. Section 3 proves a full one order superconvergence for the RT element of source problems, and furthermore, the superconvergence for the CR element and the ECR element of source problems and eigenvalue problems. Section 4 explores  asymptotic expansions of approximate eigenvalues for both the CR element and the ECR element, and proves the efficiency of  extrapolation methods. Section 5 establishes and analyzes  asymptotically exact a posteriori error estimators of eigenvalues by the CR element and the ECR element, respectively. Section 6 proposes two post-processing methods to approximate eigenvalues with high accuracy. Section 7 presents some numerical tests.

\section{Notations and Preliminaries}
\subsection{Notations}
We first introduce some basic notations. Given a nonnegative integer $k$ and a bounded domain $\Om\subset \mathbb{R}^2$ with boundary $\partial \Om$, let $W^{1,\infty}(\Om,\mathbb{R})$, $H^k(\Om,\mathbb{R})$, $\parallel \cdot \parallel_{k,\Om}$ and $|\cdot |_{k,\Om}$ denote the usual Sobolev spaces, norm, and semi-norm, respectively. And $H_0^1(\Om,\mathbb{R}) = \{u\in H^1(\Om,\mathbb{R}): u|_{\partial \Om}=0\}$. Denote the standard $L^2(\Om,\mathbb{R})$ inner product and $L^2(K,\mathbb{R})$ inner product by $(\cdot, \cdot)$ and $(\cdot, \cdot)_{0,K}$, respectively.

Suppose that $\Om\subset \mathbb{R}^2$ is a bounded polygonal domain covered exactly by a shape-regular partition $\cT_h$ into simplices. Let $|K|$ denote the volume of element $K$ and $|e|$ the length of edge $e$. Let $h_K$ denote the diameter of element $K\in \cT_h$ and $h=\max_{K\in\cT_h}h_K$. Denote the set of all interior edges and boundary edges of $\cT_h$ by $\cE_h^i$ and $\cE_h^b$, respectively, and $\cE_h=\cE_h^i\cup \cE_h^b$. For any interior edge $e=K_e^1\cap K_e^2$, we denote the element with larger global label by $K_e^1$, the one with smaller global label by $K_e^2$. Denote the corresponding unit normal vector which points from $ K_e^1 $ to $K_e^2$ by $\bold{n}_e$. Let $[\cdot]$ be jumps of piecewise functions over edge $e$, namely
$$
[v]|_e := v|_{K_e^1}-v|_{K_e^2}$$
for any piecewise function $v$. For $K\subset\R^2,\ r\in \mathbb{Z}^+$, let $P_r(K)$ be the space of all polynomials of degree not greater than $r$ on $K$. For $r\geq 1$, denote
$$
\nabla P_r(K):=\{\nabla v: v\in P_r(K)\}.
$$
Denote the piecewise gradient operator and the piecewise hessian operator by $\nabla_h$ and $\nabla_h^2$, respectively.

Let $K$ have vertices $\bold{p}_i=(p_{i1},p_{i2}),1\leq i\leq 3$ oriented counterclockwise. Denote $\{e_i\}_{i=1}^3$ the edges of element $K$, $\{l_i\}_{i=1}^3$ the edge lengths, $\{d_i\}_{i=1}^3$ the perpendicular heights, and $\{\bold{n}_i\}_{i=1}^3$ the unit outward normal vectors(see Figure \ref{fig:geometric}). Denote the second order derivatives $\frac{\partial^2 u}{\partial x_i\partial x_j}$ by $\partial_{x_ix_j} u$, $1\leq i, j\leq 2$.

\begin{figure}[!ht]
\begin{center}
\begin{tikzpicture}[xscale=8,yscale=8]
\tikzstyle{every node}=[font=\Large,scale=0.9]
\draw[-](0,0) -- (0.7,0);
\draw[-] (0,0) -- (0.5,0.4);
\draw[-] (0.7,0) -- (0.5,0.4);
\draw[-] (0.5,0) -- (0.5,0.4);
\draw[-] (0.5,0.02) -- (0.52,0.02);
\draw[-] (0.52,0.02) -- (0.52,0);
\draw[->][line width=1pt] (0.5,0) -- (0.5,-0.07);
\draw[->][line width=1pt] (0.6,0.2) -- (0.66,0.23);
\draw[->][line width=1pt] (0.25,0.2) -- (0.21,0.25);
\node[below] at (0,0) {$\bold{p}_1$};
\node[below] at (0.7,0) {$\bold{p}_2$};
\node[above] at (0.5,0.4) {$\bold{p}_3$};
\node[below] at (0.3,0) {$e_3$};
\node[right] at (0.5,-0.07) {$\bold{n}_3$};
\node[right] at (0.62,0.18) {$\bold{n}_1$};
\node at (0.26,0.28) {$\bold{n}_2$};
\node[left] at (0.5,0.15) {$d_3$};
\end{tikzpicture}
\end{center}
\caption{Paramters associated with a triangle $K$.}
\label{fig:geometric}
\end{figure}
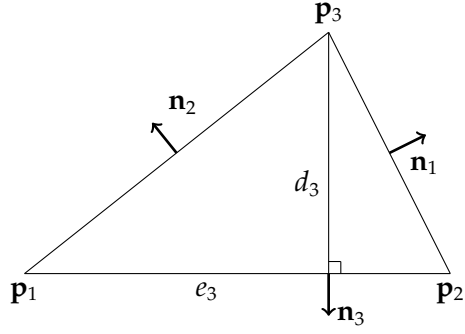

Throughout the paper, a positive constant independent of the mesh size is denoted by $C$, which refers to different values at different places. For ease of presentation, we shall use the symbol $A\lesssim B$ to denote that $A\leq CB$.
\subsection{Second order elliptic eigenvalue problems}
On a domain $\Om\subset \mathbb{R}^2$ with Lipschitz boundary, we consider a model eigenvalue problem of finding : $(\lambda, u)\in \mathbb{R}\times V$ with $\parallel u \parallel_{0,\Om}=1$ such that
\be\label{variance}
a(u, v)=\lambda(u, v) \text{\quad for any }v\in V,
\ee
where $V:= H^1_0(\Om,\mathbb{R})$. The bilinear form
$$
a(w, v):=\int_{\Om} \nabla w\cdot \nabla v \dx
$$
is symmetric, bounded, and coercive in the following sense:
$$
a(w,v)=a(v,w),\quad |a(w,v)|\lesssim \parallel w\parallel_{1,\Om}\parallel v\parallel_{1,\Om},\quad \parallel v\parallel_{1,\Om}^2\lesssim a(v,v)\text{\quad for any } w, v\in V.
$$
The eigenvalue problem \eqref{variance} has a sequence of eigenvalues
$$0<\lambda_1\leq \lambda_2\leq \lambda_3\leq ...\nearrow +\infty,$$
and the corresponding eigenfunctions
$$u_1, u_2, u_3,... ,$$
which satisfy
$$(u_i, u_j)=\delta_{ij}\ \text{ with } \delta_{ij}=\begin{cases}
0 &\quad i\neq j\\
1 &\quad i=j
\end{cases}.$$

Let $V_h$ be a nonconforming finite element approximation to $V$ over $\cT_h$. The corresponding finite element approximation of \eqref{variance} is: find $(\lambda_h, u_h)\in \mathbb{R}\times V_h$, such that $\parallel u_h\parallel_{0,\Om}=1$ and
\be\label{discrete}
a_h(u_h,v_h)=\lambda_h(u_h, v_h)\quad \text{ for any }v_h\in V_h,
\ee
where the discrete bilinear form $a_h(w_h,v_h)$ is defined elementwise as
$$
a_h(w_h,v_h):=\sum_{K\in\cT_h}\int_K \nabla_h w_h\cdot \nabla_h v_h\dx.
$$
Let $N=\text{dim }V_h$. Suppose that $ \parallel \cdot \parallel_h:=a_h(\cdot, \cdot)^{1/2}$ is a norm over the discrete space $V_h$, the discrete problem \eqref{discrete} admits a sequence of discrete eigenvalues
$$0<\lambda_{1,h}\leq \lambda_{2,h}\leq \lambda_{3,h}\leq ...\nearrow \lambda_{N,h},$$
and the corresponding eigenfunctions
$$u_{1,h}, u_{2,h},..., u_{N,h},$$
which satisfy $(u_{i,h}, u_{j,h})=\delta_{ij},\ 1\leq i,j\leq N.$

For discrete problem \eqref{discrete}, we consider the following two nonconforming elements: the CR element and the ECR element.

$\bullet$\quad  The CR element space over $\cT_h$ is defined in \cite{Crouzeix1973Conforming} by
\begin{equation*}
\begin{split}
V^{\rm CR}_h:=&\big \{v\in L^2(\Om,\mathbb{R})\big|v|_K\in P_1(K)\text{ for any }  K\in\cT_h, \int_e [v]\ds =0\text{ for any }  e\in \cE_h^i,\\
&\int_e v\ds=0\text{ for any }  e\in \cE_h^b\big\}.
\end{split}
\end{equation*}
Moreover, we define the canonical interpolation operator $\Pi^{\rm CR}_h: V\rightarrow V^{\rm CR}_h$ as follows:
\be\label{crinterpolation}
\int_e\Pi^{\rm CR}_hv\ds=\int_e v\ds\quad \text{ for any } e\in \cE_h,\ v\in V.
\ee
Denote the approximate eigenpair of \eqref{discrete} in the nonconforming space $V^{\rm CR}_h$ by $(\lambda^{\rm (CR, E)}_h,u^{\rm (CR, E)}_h)$, which satisfies $ \parallel u^{\rm (CR, E)}_h\parallel_{0,\Om}=1$.

$\bullet$\quad  The ECR element space over $\cT_h$ is defined in \cite{Hu2014Lower} by
\begin{equation*}
\begin{split}
V^{\rm ECR}_h:=&\big \{v\in L^2(\Om,\mathbb{R})\big|v|_K\in \text{ECR}(K)\text{ for any }  K\in\cT_h, \int_e [v]\ds =0\text{ for any }  e\in \cE_h^i,\\
&\int_e v\ds=0\text{ for any }  e\in \cE_h^b\big\}.
\end{split}
\end{equation*}
where $\text{ECR}(K)=P_1(K)+\text{span}\big\{x_1^2+x_2^2\big\}$. Define the canonical interpolation operator $\Pi^{\rm ECR}_h:V\rightarrow V_h^{\rm ECR}$  for any $v\in V$, by
\begin{equation}\label{ecrinterpolation}
\begin{split}
\int_e\Pi^{\rm ECR}_hv\ds&= \mathlarger{\int}_e v\ds\quad\text{ for any } e\in\cE_h,\\
\mathlarger{\int}_K \Pi^{\rm ECR}_hv\dx&= \int_K v\dx\quad\text{ for any }  K\in\cT_h.
\end{split}
\end{equation}
Denote the approximate eigenpair of \eqref{discrete} in the nonconforming space $V^{\rm ECR}_h$ by $(\lambda^{\rm (ECR, E)}_h,u^{\rm (ECR, E)}_h)$, which satisfies $ \parallel u^{\rm (ECR, E)}_h\parallel_{0,\Om}=1$. It follows from the theory of nonconforming eigenvalue approximations in \cite{Hu2014Lower} that
\begin{equation}\label{CR:est}
|\lambda-\lambda^{\rm (CR, E)}_h|+\parallel u- u^{\rm (CR, E)}_h\parallel_{0,\Om} + h^s\parallel \nabla_h (u-u^{\rm (CR, E)}_h)\parallel_{0,\Om}\lesssim h^{2s}\parallel u\parallel_{1+s,\Om},
\end{equation}
\begin{equation}\label{ECR:est}
|\lambda-\lambda^{\rm (ECR, E)}_h|+\parallel u - u^{\rm (ECR, E)}_h\parallel_{0,\Om} + h^s\parallel \nabla_h (u-u^{\rm (ECR, E)}_h)\parallel_{0,\Om}\lesssim h^{2s}\parallel u\parallel_{1+s,\Om},
\end{equation}
provided that $u\in H^{1+s}(\Om,\mathbb{R})\cap H^1_0(\Om,\mathbb{R})$, $\ 0<s\leq 1$.

For the CR element and the ECR element, the commuting property of the canonical interpolations reads
\begin{equation}\label{commuting}
\begin{split}
\int_K \nabla(w - \Pi^{\rm CR}_h w)\cdot \nabla v_h \dx &=0\quad\text{ for any } w\in V, v_h\in V^{\rm CR}_h,\\
\int_K \nabla(w - \Pi^{\rm ECR}_h w)\cdot \nabla v_h \dx &=0\quad\text{ for any }w\in V, v_h\in V^{\rm ECR}_h,
\end{split}
\end{equation}
see \cite{Crouzeix1973Conforming,Hu2014Lower} for details. For the CR element, the commuting property of the canonical interpolation operator $\Pi_h^{\rm CR}$ gives
\begin{equation}\label{commutId}
\begin{split}
\lambda-\lambda_h^{\rm (CR, E)}=&a_h(u -u_h^{\rm (CR, E)},u-u_h^{\rm (CR, E)})-2\lambda_h^{\rm (CR, E)}(u-\Pi_h^{\rm CR}u,u_h^{\rm (CR, E)})\\
&-\lambda_h^{\rm (CR, E)}(u-u_h^{\rm (CR, E)},u-u_h^{\rm (CR, E)}).
\end{split}
\end{equation}
A similar identity holds for the approximate eigenpair $(\lambda_h^{\rm (ECR, E)}, u_h^{\rm (ECR, E)})$ by the ECR element. These two identities are crucial for the analysis of extrapolation methods and asymptotically exact a posteriori error estimators.

\section{Superconvergence results}\label{sec:super}
In this section, a full one order superconvergence is proved for the RT element of the Poisson problem. Furthermore, based on the superconvergence of the RT element, the superconvergence results for the CR element and the ECR element of the Poisson problem are improved from a half order to a full one order. These results are also extended to the corresponding eigenvalue problem.

\subsection{Superconvergence of the RT element}
To begin with,  we consider the following elliptic problem: find $(\bold{\sigma}^{\rm (f,\ S)},u^{\rm (f,\ S)})\in H(\text{div},\Om,\mathbb{R}^2)\times L^2(\Om,\mathbb{R})$ such that:
\begin{equation}\label{weak}
\begin{aligned}
(\bold{\sigma}^{\rm (f,\ S)},\bold{\tau})-(u^{\rm (f,\ S)},\text{div}\bold{\tau})&=0&&\text{ for any }\bold{\tau}\in H(\text{div},\Om,\mathbb{R}^2),\\
(v,\text{div}\bold{\sigma}^{\rm (f,\ S)})&=(f,v)&&\text{ for any }v\in L^2(\Om,\mathbb{R}),
\end{aligned}
\end{equation}
where $f\in L^2(\Om,\mathbb{R})$ and
$$
H(\text{div},\Om,\mathbb{R}^2) = \{\tau\in L^2(\Om,\mathbb{R}^2),\ \text{div} \tau\in L^2(\Om,\mathbb{R})\}.
$$
One mixed finite element is the first order RT element \cite{Raviart1977A} whose shape function space is
$$
\text{RT}_K:=(P_0(K))^2+ \bold{x}P_0(K)\quad\text{ for any }K\in \cT_h.
$$
The corresponding global finite element space reads
$$
\text{RT}(\cT_h):=\big \{\tau\in H(\text{div},\Om,\mathbb{R}^2): \tau|_K\in \text{RT}_K\text{ for any }K\in \cT_h\big \}.
$$
To get a stable pair of space, the piecewise constant space is used to approximate the displacement, namely,
$$
U_{\text{RT}}:=\big \{v\in L^2(\Om,\mathbb{R}):v|_K\in P_0(K) \text{ for any }K\in \cT_h\big \}.
$$
The RT element method of \eqref{weak} seeks $(\bold{\sigma}_h^{\rm (RT,\ f)},u_h^{\rm (RT,\ f)})\in \text{RT}(\cT_h)\times U_{\text{RT}}$ such that
\begin{equation}\label{weakdis}
\begin{aligned}
(\bold{\sigma}_h^{\rm (RT,\ f)},\bold{\tau}_h)-(u_h^{\rm (RT,\ f)},\text{div}\bold{\tau}_h)&=0&&\text{ for any }\bold{\tau}_h\in \text{RT}(\cT_h),\\
(v_h,\text{div}\bold{\sigma}_h^{\rm (RT,\ f)})&=(f,v_h)&&\text{ for any }v_h\in U_{\text{RT}}.
\end{aligned}
\end{equation}
According to \cite{Brezzi2009On}, the discrete system \eqref{weakdis} has an unique solution $(\bold{\sigma}_h^{\rm (RT,\ f)},u_h^{\rm (RT,\ f)})\in \text{RT}(\cT_h)\times U_{\text{RT}}$. Meanwhile, there exist the following optimal error estimates with detailed proofs referring to \cite{Douglas1985Global}
\begin{equation*}
\begin{split}
\parallel \bold{\sigma}^{\rm (f,\ S)}-\bold{\sigma}_h^{\rm (RT,\ f)}\parallel_{0,\Om}&\lesssim h|\bold{\sigma}^{\rm (f,\ S)}|_{1,\Om},\label{RTu}\\
\parallel \text{div}(\bold{\sigma}^{\rm (f,\ S)}-\bold{\sigma}_h^{\rm (RT,\ f)})\parallel_{0,\Om}&\lesssim h|\bold{\sigma}^{\rm (f,\ S)}|_{2,\Om},
\end{split}
\end{equation*}
provided that $\sigma^{\rm (f,\ S)} \in H^2(\Om,\mathbb{R}^2)$.

The Fortin interpolation operator, which is widely used in error analysis, such as \cite{Douglas1985Global,Dur1990Superconvergence}, is defined by $\Pi^{\rm RT}_h:H^1(\Om,\mathbb{R}^2)\rightarrow \text{RT}(\cT_h)$ as
\begin{equation*}\label{def:fortin}
\int_e (\Pi^{\rm RT}_h \bold{\tau}-\bold{\tau})^T\bold{n}_e\ds=0\text{\quad for any }e\in \cE_h, \bold{\tau}\in H^1(\Om,\mathbb{R}^2).
\end{equation*}
It is proved in \cite{Raviart1977A} that for any $\tau \in H^1(\Om,\mathbb{R}^2)$,
\begin{align}
(\text{div}(\bold{\tau}-\Pi^{\rm RT}_h\bold{\tau}),v_h)&=0\quad \text{for any }v_h\in U_{\text{RT}},\label{fortinId}\\
\parallel \bold{\tau}- \Pi^{\rm RT}_h\bold{\tau}\parallel_{0,\Om}&\lesssim h|\bold{\tau}|_{1,\Om}.\label{fortin}
\end{align}
It follows from \eqref{weakdis} and \eqref{fortinId} that
\begin{equation*}\label{solinter}
\text{div} \bold{\sigma}_h^{\rm (RT,\ f)}= \text{div} \Pi^{\rm RT}_h\bold{\sigma}^{\rm (f,\ S)}.
\end{equation*}
Therefore, $\bold{\sigma}_h^{\rm (RT,\ f)} - \Pi^{\rm RT}_h\bold{\sigma}^{\rm (f,\ S)}\in \rm RT(\cT_h)$ is divergence free, and is a piecewise constant vector field. Hence, a substitution of $\bold{\tau}_h=\bold{\sigma}_h^{\rm (RT,\ f)} - \Pi^{\rm RT}_h\bold{\sigma}^{\rm (f,\ S)}$ into \eqref{weak} and \eqref{weakdis} yields
\be\label{sigtosigh}
(\bold{\sigma}^{\rm (f,\ S)},\bold{\sigma}_h^{\rm (RT,\ f)}-\Pi^{\rm RT}_h\bold{\sigma}^{\rm (f,\ S)})=(\bold{\sigma}_h^{\rm (RT,\ f)},\bold{\sigma}_h^{\rm (RT,\ f)}-\Pi^{\rm RT}_h\bold{\sigma}^{\rm (f,\ S)}).
\ee

Throughout this section, the superconvergence result of the gradient recovery operator in \cite{Hu2016Superconvergence} requires triangulations to be uniform:
\begin{Def}
A triangulation $\cT_h$ of $\Om$ is said to be uniform if any two adjacent triangles of $\cT_h$ form a parallelogram.
\end{Def}

For any triangle $K\in\cT_h$, from the three outer unit normal vectors, denote the two which are closest to orthogonal by $\bold{f}_1$ and $\bold{f}_2$. This procedure is in general not unique, however, only the directions of vectors are focused, thus there will be no restriction.

For each $i=1,\ 2$, denote a parallelogram, which consists of two triangles sharing a side with normal $\bold{f}_i$, by $N_{\bold{f}_i}$. We partition the domain $\Om$ into those parallelograms $N_{\bold{f}_i} $ and some resulted boundary triangles, and denote these boundary triangles by $K_{\bold{f}_i} $. In an element $K$, we denote the edge to which the unit normal vector is $\bold{f}_i$ by $e_{\bold{f}_i}$, the length of $e_{\bold{f}_i}$ by $h_{\bold{f}_i}$, and the unit tangent vector of $e_{\bold{f}_i}$ with counterclockwise by $\bold{t}_{\bold{f}_i}$. We denote the two endpoints of the edge $e_{\bold{f}_i} $ by $\bold{p}_{\bold{f}_i}^1$ and $\bold{p}_{\bold{f}_i}^2$, and $\bold{p}_{\bold{f}_i}^1\bold{p}_{\bold{f}_i}^2 =h_{\bold{f}_i}\bold{t}_{\bold{f}_i}$.
Define
$$
\cP_b:=\big \{\bold{p} \in \partial \Om: \bold{p} \text{ is a vertex of } K_{\bold{f}_i}\big \}.
$$
Decompose the set $\cP_b$ into two parts $ \cP_b=\cP_b^1\cup \cP_b^2$, where $\cP_b^1$ is the set of vertices of the domain $\Om$, and $\cP_b^2$ refers to the remaining vertices. For any vertex $\bold{p}\in \cP_b^1$, denote the unique boundary triangle $K_{\bold{f}_i} $ by $K_\bold{p}$, and for any vertex $\bold{p}\in \cP_b^2$, denote the two boundary triangles $K_{\bold{f}_i} $ sharing $\bold{p}$ by $K^l_\bold{p}$ and $K^r_\bold{p}$, where
$$
K^r_\bold{p}=\{\bold{x}+h_{\bold{f}_i}\bold{t}_{\bold{f}_i}:\bold{x}\in K^l_\bold{p}\}.
$$
For any $\bold{p}\in \cP_b^2$, let $\omega_{\bold{p}}$ be the trapezoid which is made up of three elements and $\bold{p}$ is a midpoint of its edge, see Figure \ref{fig:triangulation}. And $|\cP_b^1|=\kappa$ is the number of the elements in $\cP_b^1$, it is known that $\kappa $ is a fixed number independent of $h$. Figure \ref{fig:triangulation} shows an example of the definitions and notations concerning a triangulation.

\begin{figure}[!ht]
\begin{center}
\begin{tikzpicture}[xscale=6,yscale=6]
\draw[-] (0,0) -- (1.75,0);
\draw[-] (-0.125,0.9) -- (1.625,0.9);
\draw[-] (0.375,0.9) -- (0,0);
\draw[-] (0.625,0.9) -- (0.25,0);
\draw[-] (0.875,0.9) -- (0.5,0);
\draw[-] (1.125,0.9) -- (0.75,0);
\draw[-] (1.375,0.9) -- (1,0);
\draw[-] (1.625,0.9) -- (1.25,0);
\draw[-] (1.75,0.6) -- (1.5,0);
\draw[-] (0.125,0.9) -- (0.5,0);
\draw[-] (0.375,0.9) -- (0.75,0);
\draw[-] (0.625,0.9) -- (1,0);
\draw[-] (0.875,0.9) -- (1.25,0);
\draw[-] (1.125,0.9) -- (1.5,0);
\draw[-] (1.375,0.9) -- (1.75,0);
\draw[-] (1.625,0.9) -- (1.875,0.3);
\draw[-] (-0.125,0.9) -- (0.25,0);
\draw[-] (0.125,0.9) -- (0,0.6);
\draw[-] (0,0.6) -- (1.75,0.6);
\draw[-] (0.125,0.3) -- (1.875,0.3);
\draw[-] (1.75,0) -- (1.875,0.3);

\draw[-] (0.15,0.24) -- (0.275,0.54);
\draw[-] (0.175,0.18) -- (0.3,0.48);
\draw[-] (0.2,0.12) -- (0.325,0.42);
\draw[-] (0.225,0.06) -- (0.35,0.36);

\draw[-] (0.675,0.3) -- (0.8,0.6);
\draw[-] (0.725,0.3) -- (0.85,0.6);
\draw[-] (0.775,0.3) -- (0.9,0.6);
\draw[-] (0.825,0.3) -- (0.95,0.6);

\draw[-] (-0.075,0.9) -- (0.025,0.66);
\draw[-] (-0.025,0.9) -- (0.05,0.72);
\draw[-] (0.025,0.9) -- (0.075,0.78);
\draw[-] (0.075,0.9) -- (0.1,0.84);

\draw[-] (0.175,0.9) -- (0.275,0.66);
\draw[-] (0.225,0.9) -- (0.3,0.72);
\draw[-] (0.275,0.9) -- (0.325,0.78);
\draw[-] (0.325,0.9) -- (0.35,0.84);

\draw[-] (0.425,0.9) -- (0.525,0.66);
\draw[-] (0.475,0.9) -- (0.55,0.72);
\draw[-] (0.525,0.9) -- (0.575,0.78);
\draw[-] (0.575,0.9) -- (0.6,0.84);

\draw[-] (0.675,0.9) -- (0.775,0.66);
\draw[-] (0.725,0.9) -- (0.8,0.72);
\draw[-] (0.775,0.9) -- (0.825,0.78);
\draw[-] (0.825,0.9) -- (0.85,0.84);

\draw[-] (0.925,0.9) -- (1.025,0.66);
\draw[-] (0.975,0.9) -- (1.05,0.72);
\draw[-] (1.025,0.9) -- (1.075,0.78);
\draw[-] (1.075,0.9) -- (1.1,0.84);

\draw[-] (1.175,0.9) -- (1.275,0.66);
\draw[-] (1.225,0.9) -- (1.3,0.72);
\draw[-] (1.275,0.9) -- (1.325,0.78);
\draw[-] (1.325,0.9) -- (1.35,0.84);

\draw[-] (1.425,0.9) -- (1.525,0.66);
\draw[-] (1.475,0.9) -- (1.55,0.72);
\draw[-] (1.525,0.9) -- (1.575,0.78);
\draw[-] (1.575,0.9) -- (1.6,0.84);

\draw[-] (0.775,0.06) -- (1.225,0.06);
\draw[-] (0.8,0.12) -- (1.2,0.12);
\draw[-] (0.825,0.18) -- (1.175,0.18);
\draw[-] (0.85,0.24) -- (1.15,0.24);
\node at (1,0.18) {$\omega_\bold{p}$};

\draw[ultra thick,->] (0.5,0) -- (0.75,0);
\node[above] at (0.63,0) {$h_{\bold{f}_1}\bold{t}_{\bold{f}_1}$};
\draw[ultra thick,->] (0.75,0) -- (0.625,0.3);
\node at (0.75,0.2) {$h_{\bold{f}_2}\bold{t}_{\bold{f}_2}$};

\node[below] at (0.5,0) {$ \bold{p}_{\bold{f}_1}^1$};
\node[below] at (0.75,0) {$ \bold{p}_{\bold{f}_1}^2$};
\node at (0,0.8) {$K_\bold{{f_1}}$};
\node at (0.25,0.8) {$K_\bold{{f_1}}$};
\node at (0.5,0.8) {$K_\bold{{f_1}}$};
\node at (0.75,0.8) {$K_\bold{{f_1}}$};
\node at (1,0.8) {$K_\bold{{f_1}}$};
\node at (1.25,0.8) {$K_\bold{{f_1}}$};
\node at (1.5,0.8) {$K_\bold{{f_1}}$};

\node at (0.25,0.4) {$N_\bold{{f_1}}$};
\node at (0.78,0.4) {$N_\bold{{f_2}}$};

\draw[->] (0.625,0.6) -- (0.625,0.67);
\draw[->] (0.5625,0.45) -- (0.495,0.422);
\node at (0.68,0.65) {$\bold{f_1}$};
\node at (0.555,0.4) {$\bold{f_2}$};

\fill(0,0) circle(0.5pt);
\fill(0.25,0) circle(0.3pt);
\fill(0.5,0) circle(0.3pt);
\fill(0.75,0) circle(0.3pt);
\fill(1,0) circle(0.3pt);
\fill(1.25,0) circle(0.3pt);
\fill(1.5,0) circle(0.3pt);
\fill(1.75,0) circle(0.5pt);

\fill(-0.125,0.9) circle(0.5pt);
\fill(0.125,0.9) circle(0.3pt);
\fill(0.375,0.9) circle(0.3pt);
\fill(0.625,0.9) circle(0.3pt);
\fill(0.875,0.9) circle(0.3pt);
\fill(1.125,0.9) circle(0.3pt);
\fill(1.375,0.9) circle(0.3pt);
\fill(1.625,0.9) circle(0.5pt);

\node[below] at (1.05,0) {$\bold{p}\in \cP_b^2$};
\node[below] at (1.75,0) {$\bold{p}\in \cP_b^1$};
\node at (1.65,0.1) {$K_\bold{p}$};
\node at (0.9,0.1) {$K^l_\bold{p}$};
\node at (1.15,0.1) {$K^r_\bold{p}$};
\end{tikzpicture}
\caption{An uniform triangulation of $\Om$.}
\label{fig:triangulation}
\end{center}
\end{figure}
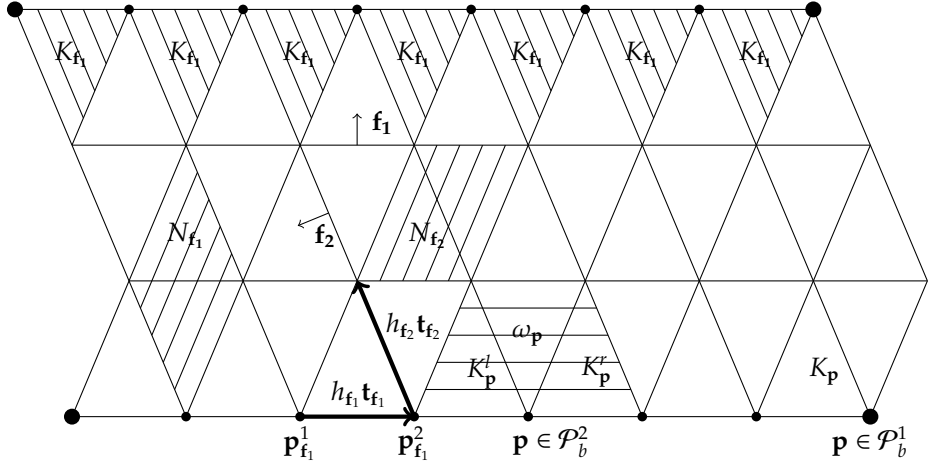

We will need some results on Sobolev spaces. Denote the subset of the points in $\Om$ having distance less than $h$ from the boundary by $\partial_h \Om$:
$$
\partial_h\Om=\{\bold{x}\in\Om:\exists \bold{y}\in\partial \Om \text{ such that dist}(\bold{x},\bold{y})\leq h\}.
$$
We have the following result, see \cite{Brandts1994Superconvergence} and the references therein.
\begin{Lm}\label{bdtodomain}
For $v\in H^s(\Om,\mathbb{R})$, where $0\leq s\leq \frac{1}{2}$,
$$
\parallel v \parallel_{0,\partial_h \Om}\lesssim h^s\parallel v\parallel_{s,\Om}.
$$
\end{Lm}
We recall the following discrete Helmholtz decomposition and refer to \cite{Jan2000Superconvergence} for more details.
\begin{Lm}\label{lemma2}
For any function $\bold{\tau}_h\in \rm RT(\cT_h)$ which satisfies
$
\text{div } \bold{\tau}_h=0,
$
$$
\bold{\tau}_h\in \bold{curl} \cP_1,
$$
where $ \cP_1=\{v\in H^1_0(\Om,\mathbb{R}):v|_K\in P_1(K)\ \text{ for any }K\in\cT_h\}$.
\end{Lm}

Assume that the triangulation $\cT_h$ is uniform. Suppose that the solution of the problem \eqref{weak} satisfies $\sigma^{\rm (f,\ S)}\in H^{\frac{5}{2}}(\Om,\mathbb{R}^2)$. Define a matrix $F$, the transportation of which has the two unit normal vectors $\bold{f}_1$ and $\bold{f}_2$ as columns. Denote the canonical basis vectors of $\mathbb{R}^2$ in respectively the $x_1$- and $x_2$-direction by $\bold{e}_1$ and $\bold{e}_2$. By \eqref{sigtosigh},
\begin{equation*}
\begin{split}
\parallel \bold{\sigma}_h^{\rm (RT,\ f)}-\Pi^{\rm RT}_h\bold{\sigma}^{\rm (f,\ S)}\parallel_{0,\Om}^2&=\big (F(\bold{\sigma}_h^{\rm (RT,\ f)}-\Pi^{\rm RT}_h\bold{\sigma}^{\rm (f,\ S)}),F^{-T}(\bold{\sigma}^{\rm (f,\ S)}-\Pi^{\rm RT}_h\bold{\sigma}^{\rm (f,\ S)})\big )\\
&=\sum_{K\in\cT_h} \int_K \big (F(\bold{\sigma}_h^{\rm (RT,\ f)}-\Pi^{\rm RT}_h\bold{\sigma}^{\rm (f,\ S)})\big )^TF^{-T}\big (\bold{\sigma}^{\rm (f,\ S)}-\Pi^{\rm RT}_h\bold{\sigma}^{\rm (f,\ S)}\big )\dx\\
&=\sum_{i,j=1}^2\bold{I}_{ij},
\end{split}
\end{equation*}
where
$$
\bold{I}_{ij}=\sum_{K\in\cT_h} \int_K \bold{e}_i^T F(\bold{\sigma}_h^{\rm (RT,\ f)}-\Pi^{\rm RT}_h\bold{\sigma}^{\rm (f,\ S)})(\bold{e}_i^TF^{-T}\bold{e}_j)(\bold{\sigma}^{\rm (f,\ S)}-\Pi^{\rm RT}_h\bold{\sigma}^{\rm (f,\ S)})^T\bold{e}_j\dx.
$$
For simplicity, only the sum $\bold{I}_{11}$ is considered here. Let $\Om$ be partitioned into parallelograms $N_{\bold{f}_1}$ and the remaining boundary triangles $K_{\bold{f}_1}$. Since $\bold{\sigma}_h^{\rm (RT,\ f)}-\Pi^{\rm RT}_h\bold{\sigma}^{\rm (f,\ S)}$ is piecewise constant, the sum $\bold{I}_{11}$ can be written as a sum over parallelograms $N_{\bold{f}_1}$ and boundary triangles $K_{\bold{f}_1}$:
\be\label{totalterm}
|\bold{I}_{11}|\leq |\bold{I}_{11}^1|+|\bold{I}_{11}^2|,
\ee
where
$$
\bold{I}_{11}^1=(\bold{e}_1^TF^{-T}\bold{e}_1)\sum_{N_{\bold{f}_1}}\int_{N_{\bold{f}_1}}\bold{e}_1^T F(\bold{\sigma}_h^{\rm (RT,\ f)}-\Pi^{\rm RT}_h\bold{\sigma}^{\rm (f,\ S)})(\bold{\sigma}^{\rm (f,\ S)}-\Pi^{\rm RT}_h\bold{\sigma}^{\rm (f,\ S)})^T\bold{e}_1\dx,
$$
\be\label{I2original}
\bold{I}_{11}^2=(\bold{e}_1^TF^{-T}\bold{e}_1)\sum_{K_{\bold{f}_1}}\bold{e}_1^T F(\bold{\sigma}_h^{\rm (RT,\ f)}-\Pi^{\rm RT}_h\bold{\sigma}^{\rm (f,\ S)})\int_{K_{\bold{f}_1}} (\bold{\sigma}^{\rm (f,\ S)}-\Pi^{\rm RT}_h\bold{\sigma}^{\rm (f,\ S)})^T\bold{e}_1\dx.
\ee
Note that $\bold{e}_1^TF(\bold{\sigma}_h^{\rm (RT,\ f)}-\Pi^{\rm RT}_h\bold{\sigma}^{\rm (f,\ S)})=(\bold{\sigma}_h^{\rm (RT,\ f)}-\Pi^{\rm RT}_h\bold{\sigma}^{\rm (f,\ S)})^T\bold{f}_1$ is the normal component of $\bold{\sigma}_h^{\rm (RT,\ f)}-\Pi^{\rm RT}_h\bold{\sigma}^{\rm (f,\ S)}$ to the shared side of the two triangles forming a parallelogram $N_{\bold{f}_1}$. Thus, $\bold{e}_1^TF(\bold{\sigma}_h^{\rm (RT,\ f)}-\Pi^{\rm RT}_h\bold{\sigma}^{\rm (f,\ S)})$ is constant on $N_{\bold{f}_1}$, and therefore, leads to the following superconvergence property \cite{Brandts1994Superconvergence}:
\be\label{1termfinal}
|\bold{I}_{11}^1|\lesssim h^2\parallel \bold{\sigma}_h^{\rm (RT,\ f)}-\Pi^{\rm RT}_h\bold{\sigma}^{\rm (f,\ S)}\parallel_{0,\Om}|\bold{\sigma}^{\rm (f,\ S)}|_{2,\Om}.
\ee
For the sum $\bold{I}_{11}^2$ of boundary terms, the analysis in \cite{Brandts1994Superconvergence} showed
\be\label{2termI11}
|\bold{I}_{11}^2|\lesssim h^{3/2}\parallel \bold{\sigma}_h^{\rm (RT,\ f)}-\Pi^{\rm RT}_h\bold{\sigma}^{\rm (f,\ S)}\parallel_{0,\Om}|\bold{\sigma}^{\rm (f,\ S)}|_{\frac{3}{2},\Om}.
\ee
The estimate \eqref{2termI11} is resulted from a direct application of Lemma \ref{bdtodomain}. Since the estimate in Lemma \ref{bdtodomain} can not be improved as showed by a counter example in \cite{J1972Non}, it is very difficult to improve the factor in \eqref{2termI11} from $h^{3/2}$ to $h^2$ following that analysis.

A new analysis for a full one order superconvergence of the RT element is presented in the following. The main idea here is to employ a discrete Helmholtz decomposition of $\bold{\sigma}_h^{\rm (RT,\ f)}-\Pi^{\rm RT}_h\bold{\sigma}^{\rm (f,\ S)}$. In particular, it allows for some vital cancellation between the boundary terms in $\bold{I}_{11}^2$ sharing a common vertex.

Firstly, we present the following property of the interpolation operator $\Pi^{\rm RT}_h$.
\begin{Lm}\label{linearsigma}
For any $\bold{p}\in\cP_b^2$, $K_{\bold{p}}^l$, $K_{\bold{p}}^r\in K_{\bold{f}_i}$, $i=1,\ 2$, if $\bold{\sigma}$ is linear on the patch $\omega_{\bold{p}}$,
\begin{equation*}
\begin{split}
\int_{K_{\bold{p}}^l}(\bold{\sigma}-\Pi^{\rm RT}_h\bold{\sigma})\dx= \int_{K_{\bold{p}}^r}(\bold{\sigma}-\Pi^{\rm RT}_h\bold{\sigma})\dx.
\end{split}
\end{equation*}
\end{Lm}
\begin{proof}
Denote the centroid, the vertices and the edges of element $K_{\bold{p}}^l$ by $\bold{M}_{K_{\bold{p}}^l }$, $\{\bold{p}_i^l\}_{i=1}^3$ and $\{e_i^l\}_{i=1}^3$, and those of element $K_{\bold{p}}^r$ by $\bold{M}_{K_{\bold{p}}^r}$, $\{\bold{p}_i^r\}_{i=1}^3$ and $\{e_i^r\}_{i=1}^3$. For edge $e_i^l$, denote the midpoint, the unit outward normal vector and the perpendicular height by $\bold{m}_i^l$, $\bold{n}^l_i$ and $d^l_i$, respectively. And denote those of edge $e_i^r$ by $\bold{m}_i^r$, $\bold{n}^r_i$ and $d^r_i$, respectively. Let $\phi_i^l=\frac{1}{d_i^l}(\bold{x}-\bold{p}_i^l)$ and $\phi_i^r=\frac{1}{d_i^r}(\bold{x}-\bold{p}_i^r),\ 1\leq i\leq 3$, which  are the basis functions of the RT element on elements $K_{\bold{p}}^l$ and $K_{\bold{p}}^r$, respectively.

Since $\bold{\sigma}$ is linear on the patch $\omega_{\bold{p}}$, $
\bold{\sigma}(\bold{x})= \bold{\sigma}(\bold{p}) + \nabla \bold{\sigma}\cdot (\bold{x} - \bold{p}).
$ Thus,
$$
\bold{\sigma}(\bold{x})-\Pi_h^{\rm RT}\bold{\sigma}(\bold{x}) = (I - \Pi^{\rm RT}_h)\big (\nabla \bold{\sigma}\cdot(\bold{x}-\bold{p})\big ).
$$
The fact that
$$
\int_{K_{\bold{p}}^l } (I - \Pi^{\rm RT}_h)\big (\nabla \bold{\sigma}\cdot( \bold{M}_{K_{\bold{p}}^l }-\bold{p})\big )\dx =0\ \text{ and } \int_{K_{\bold{p}}^l } \nabla \bold{\sigma}\cdot(\bold{x}- \bold{M}_{K_{\bold{p}}^l })\dx =0
$$
leads to
\be\label{onlyPiRT}
\int_{K_{\bold{p}}^l} \bold{\sigma}(\bold{x})-\Pi_h^{\rm RT}\bold{\sigma}(\bold{x})\dx = -\int_{K_{\bold{p}}^l} \Pi^{\rm RT}_h\big (\nabla \bold{\sigma}\cdot(\bold{x}-\bold{M}_{K_{\bold{p}}^l })\big )\dx.
\ee
Note that $\nabla\sigma|_{K_{\bold{p}}^l }=\nabla \sigma|_{K_{\bold{p}}^r}$, $\bold{n}^l_i=\bold{n}^r_i$ and $\bold{m}_i^l-\bold{M}_{K_{\bold{p}}^l}=\bold{m}_i^r-\bold{M}_{K_{\bold{p}}^r }$. Thus
$$
(\nabla \bold{\sigma}\cdot(\bold{m}_i^l-\bold{M}_{K_{\bold{p}}^l }))^T\bold{n}_i^l=(\nabla \bold{\sigma}\cdot(\bold{m}_i^r-\bold{M}_{K_{\bold{p}}^r }))^T\bold{n}_i^r.
$$
Since $\int_{K_{\bold{p}}^l} \phi_i^l\dx = \int_{K_{\bold{p}}^r} \phi_i^r\dx$, this and \eqref{onlyPiRT} lead to
$$
\int_{K_{\bold{p}}^l}(\bold{\sigma}-\Pi^{\rm RT}_h\bold{\sigma})\dx= \int_{K_{\bold{p}}^r}(\bold{\sigma}-\Pi^{\rm RT}_h\bold{\sigma})\dx,
$$
which completes the proof.
\end{proof}

By employing the discrete Helmholtz decomposition, the estimate of the term $ \bold{I}_{11}$ in \cite{Brandts1994Superconvergence} is improved in the following lemma.
\begin{Lm}
Suppose that $(\bold{\sigma}^{\rm (f,\ S)},u^{\rm (f,\ S)})$ is the solution of \eqref{weak} with $\sigma^{\rm (f,\ S)}\in H^{\frac{5}{2}}(\Om,\mathbb{R}^2)$, $(\bold{\sigma}_h^{\rm (RT,\ f)},u_h^{\rm (RT,\ f)})$ is the solution of \eqref{weakdis} on an uniform triangulation $\cT_h$. It holds that
\begin{equation*}\label{2termfinal}
|\bold{I}_{11}|\lesssim h^2 (| \bold{\sigma}^{\rm (f,\ S)}|_{\frac{5}{2},\Om}+\kappa |\ln h|^{1/2}|\bold{\sigma}^{\rm (f,\ S)}|_{1,\infty,\Om})\parallel \bold{\sigma}_h^{\rm (RT,\ f)}-\Pi^{\rm RT}_h\bold{\sigma}^{\rm (f,\ S)}\parallel_{0,\Om}.
\end{equation*}
\end{Lm}
\begin{proof}
By Lemma \ref{lemma2}, there exists $w_h\in \cP_1$ such that
$$\bold{\sigma}_h^{\rm (RT,\ f)}-\Pi^{\rm RT}_h\bold{\sigma}^{\rm (f,\ S)}=\bold{curl} w_h \in \big(U_{\rm RT}\big)^2.$$
The term $\bold{I}_{11}^2$ in \eqref{I2original} reads
\be\label{2term1}
\bold{I}_{11}^2=(\bold{e}_1^TF^{-T}\bold{e}_1)\sum_{K_{\bold{f}_1}} \bold{e}_1^TF\bold{curl} w_h\int_{K_{\bold{f}_1}}(\bold{\sigma}^{\rm (f,\ S)}-\Pi^{\rm RT}_h\bold{\sigma}^{\rm (f,\ S)})^T\bold{e}_1\dx.
\ee
Since
\begin{equation}\label{2term2}
\bold{e}_1^TF\bold{curl}w_h=\frac{1}{h_{\bold{f}_1}}\int_{e_{\bold{f}_1}}\nabla w_h\cdot \bold{t}_{\bold{f}_1}\ds=\frac{w_h(\bold{p}_{\bold{f}_1}^2)-w_h(\bold{p}_{\bold{f}_1}^1)}{h_{\bold{f}_1}},
\end{equation}
a substitution of \eqref{2term2} into \eqref{2term1} leads to
\be\label{2term3}
\begin{split}
|\bold{I}_{11}^2|\lesssim &\big|\sum_{\bold{p}\in\cP_b^2} \frac{w_h(\bold{p})}{h_{\bold{f}_1}}\big(\int_{K_{\bold{p}}^l}(\bold{\sigma}^{\rm (f,\ S)}-\Pi^{\rm RT}_h\bold{\sigma}^{\rm (f,\ S)})^T\bold{e}_1\dx- \int_{K_{\bold{p}}^r}(\bold{\sigma}^{\rm (f,\ S)}-\Pi^{\rm RT}_h\bold{\sigma}^{\rm (f,\ S)})^T\bold{e}_1\dx\big)\big|\\
& + \sum_{\bold{p}\in\cP_b^1} \big| \frac{w_h(\bold{p})}{h_{\bold{f}_1}}\int_{K_{\bold{p}}}(\bold{\sigma}^{\rm (f,\ S)}-\Pi^{\rm RT}_h\bold{\sigma}^{\rm (f,\ S)})^T\bold{e}_1\dx\big|.
\end{split}
\ee
By Lemma \ref{linearsigma} and the Bramble-Hilbert lemma,
\be\label{2term4}
\begin{split}
\big |\int_{K_{\bold{p}}^l}(\bold{\sigma}^{\rm (f,\ S)}-\Pi^{\rm RT}_h\bold{\sigma}^{\rm (f,\ S)})^T\bold{e}_1\dx- \int_{K_{\bold{p}}^r}(\bold{\sigma}^{\rm (f,\ S)}-\Pi^{\rm RT}_h\bold{\sigma}^{\rm (f,\ S)})^T\bold{e}_1\dx\big |\lesssim h^3|\bold{\sigma}^{\rm (f,\ S)}|_{2,\omega_{\bold{p}}}.
\end{split}
\ee
A substitution of \eqref{2term4} and the Cauchy-Schwarz inequality into \eqref{2term3} yields
\begin{equation*}
\begin{split}
|\bold{I}_{11}^2|\lesssim &h\big(\sum_{\bold{p}\in\cP_b^2} \parallel w_h\parallel_{0,\omega_{\bold{p}}}^2\big)^{1/2}\big(\sum_{\bold{p}\in\cP_b^2}|\bold{\sigma}^{\rm (f,\ S)}|_{2,\omega_{\bold{p}}}^2\big)^{1/2}\\
 &+ h^2\big(\sum_{\bold{p}\in\cP_b^1} \parallel w_h\parallel_{0,\infty,K_{\bold{p}}}^2\big)^{1/2}\big(\sum_{\bold{p}\in\cP_b^1}|\bold{\sigma}^{\rm (f,\ S)}|_{1,\infty,K_{\bold{p}}}^2\big)^{1/2} \\
\lesssim &h \parallel w_h\parallel_{0,\partial_h \Om} | \bold{\sigma}^{\rm (f,\ S)}|_{2,\partial_h\Om}+\kappa h^2 |\bold{\sigma}^{\rm (f,\ S)}|_{1,\infty,\Om} \parallel w_h\parallel_{0,\infty,h}.
\end{split}
\end{equation*}
Lemma \ref{bdtodomain} implies that
\begin{equation*}
h \parallel w_h\parallel_{0,\partial_h \Om} | \bold{\sigma}^{\rm (f,\ S)}|_{2,\partial_h\Om}\lesssim h^2 \parallel w_h\parallel_{\frac{1}{2},\Om} | \bold{\sigma}^{\rm (f,\ S)}|_{\frac{5}{2},\Om}\lesssim h^2 \parallel w_h\parallel_{1,h} | \bold{\sigma}^{\rm (f,\ S)}|_{\frac{5}{2},\Om}.
\end{equation*}
Since $ \parallel w_h\parallel_{0,\infty,h}\lesssim |\ln h|^{1/2}\parallel w_h\parallel_{1,h}$,
\be\label{2termfinal}
|\bold{I}_{11}^2|\lesssim (h^2 | \bold{\sigma}^{\rm (f,\ S)}|_{\frac{5}{2},\Om}+\kappa h^2|\ln h|^{1/2}|\bold{\sigma}^{\rm (f,\ S)}|_{1,\infty,\Om}) \parallel w_h\parallel_{1,h}.
\ee
A substitution of \eqref{1termfinal} and \eqref{2termfinal} into \eqref{totalterm} concludes
$$
|\bold{I}_{11}|\lesssim (h^2 | \bold{\sigma}^{\rm (f,\ S)}|_{\frac{5}{2},\Om}+\kappa h^2|\ln h|^{1/2}|\bold{\sigma}^{\rm (f,\ S)}|_{1,\infty,\Om})\parallel \bold{\sigma}_h^{\rm (RT,\ f)}-\Pi^{\rm RT}_h\bold{\sigma}^{\rm (f,\ S)}\parallel_{0,\Om},
$$
which completes the proof.
\end{proof}

Similar arguments for the sums $ \bold{I}_{12}$, $ \bold{I}_{21}$ and $ \bold{I}_{22}$ prove a full one order superconvergence for the RT element as follows.
\begin{Th}\label{Lm:bdRT}
Suppose that $(\bold{\sigma}^{\rm (f,\ S)},u^{\rm (f,\ S)})$ is the solution of \eqref{weak} with $\sigma^{\rm (f,\ S)}\in H^{\frac{5}{2}}(\Om,\mathbb{R}^2)$, and $(\bold{\sigma}_h^{\rm (RT,\ f)},u_h^{\rm (RT,\ f)})$ is the solution of \eqref{weakdis} on an uniform triangulation $\cT_h$. It holds that
$$
\parallel \bold{\sigma}_h^{\rm (RT,\ f)}-\Pi^{\rm RT}_h\bold{\sigma}^{\rm (f,\ S)}\parallel_{0,\Om}\lesssim h^2 \big (| \bold{\sigma}^{\rm (f,\ S)}|_{\frac{5}{2},\Om}+\kappa |\ln h|^{1/2}|\bold{\sigma}^{\rm (f,\ S)}|_{1,\infty,\Om}\big ).
$$
\end{Th}

\subsection{Superconvergence of the CR element and the ECR element}
A full one order superconvergence of the CR element and the ECR element follows from a special relation between the RT element and the CR element, and the equivalence between the RT element and the ECR element, respectively.

A post-processing mechanism is analyzed in \cite{Hu2016Superconvergence} for the superconvergence of the CR element. Given $\textbf{q}\in \rm RT(\cT_h)$, define function $K_h \textbf{q}\in V_h^{\rm CR }\times V^{\rm CR}_h$ as follows.
\begin{Def}\label{Def:R}
1.For each interior edge $e\in\cE_h^i$, the elements $K_e^1$ and $K_e^2$ are the pair of elements sharing $e$. Then the value of $K_h \textbf{q}$ at the midpoint $\textbf{m}_e$ of $e$ is
$$
K_h \textbf{q}(\textbf{m}_e)=\frac{1}{2}\big(\textbf{q}|_{K_e^1}(\textbf{m}_e)+\textbf{q}|_{K_e^2}(\textbf{m}_e)\big).
$$

2.For each boundary edge $e\in\cE_h^b$, let $K$ be the element having $e$ as an edge, and $K'$ be an element sharing an edge $e'\in\cE_h^i$ with $K$. Let $e''$ denote the edge of $K'$ that does not intersect with $e$, and $\textbf{m}$, $\textbf{m}'$ and $\textbf{m}''$ be the midpoints of the edges $e$, $e'$ and $e''$, respectively. Then the value of $K_h \textbf{q}$ at the point $\textbf{m}$ is
$$
K_h \textbf{q}(\textbf{m})=2K_h \textbf{q}(\textbf{m}')-K_h \textbf{q}(\textbf{m}'').
$$
\begin{center}
\begin{tikzpicture}[xscale=2.5,yscale=2.5]
\draw[-] (-0.5,0) -- (2,0);
\draw[-] (0,0) -- (0.5,1);
\draw[-] (0.5,1) -- (2,1);
\draw[-] (0.5,1) -- (1.5,0);
\draw[-] (2,1) -- (1.5,0);
\node[below, right] at (1,0.5) {\textbf{m}'};
\node[above] at (1.25,1) {\textbf{m}''};
\node[below] at (0.75,0) {\textbf{m}};
\node at (0.7,0.4) {K};
\node at (1.4,0.75) {K'};
\node at (1,0) {e};
\node at (1.4,0.2) {e'};
\node at (1.7,1) {e''};
\node at (0.3,-0.1) {$\partial \Om $};
\fill(1,0.5) circle(0.5pt);
\fill(1.25,1) circle(0.5pt);
\fill(0.75,0) circle(0.5pt);
\end{tikzpicture}
\end{center}
\end{Def}
The Poisson problem is to find $u^{\rm (f,\ S)}\in H^1_0(\Om,\mathbb{R})$ such that
\be\label{Poisson}
(\nabla u^{\rm (f,\ S)},\nabla v) = (f,v)\quad \text{ for any } v\in H^1_0(\Om,\mathbb{R}),
\ee
where $f\in L^2(\Om,\mathbb{R})$. The CR element method of \eqref{Poisson} seeks $u_h^{\rm (CR,\ f)}\in V^{\rm CR}_h$ such that
\be\label{Poisson:disCR}
(\nabla_h u_h^{\rm (CR,\ f)},\nabla_h v_h) = (f,v_h)\quad \text{ for any } v_h\in V^{\rm CR}_h.
\ee
The ECR element method of \eqref{Poisson} seeks $u_h^{\rm (ECR,\ f)}\in V^{\rm ECR}_h$ such that
\begin{equation*}\label{Poisson:disECR}
(\nabla_h u_h^{\rm (ECR,\ f)},\nabla_h v_h) = (f,v_h)\quad \text{ for any } v_h\in V^{\rm ECR}_h.
\end{equation*}
Due to the superconvergence result of the RT element in Theorem \ref{Lm:bdRT} and the special relation between the RT element and the CR element \cite{Marini1985An}, the superconvergence of the CR element for \eqref{Poisson} can be improved from a half order to a full one order following the analysis in \cite{Hu2016Superconvergence}.

\begin{Th}\label{superbd}
Suppose that $u^{\rm (f,\ S)}\in H^{\frac{7}{2}}(\Om,\mathbb{R})\cap H^1_0(\Om,\mathbb{R})$ is the solution of \eqref{Poisson}, $u_h^{\rm (CR,\ f)}$ is the solution of \eqref{Poisson:disCR} by the CR element on an uniform triangulation $\cT_h$, and $f\in W^{1,\infty}(\Om,\mathbb{R})$. It holds that
$$
\parallel \nabla u^{\rm (f,\ S)}-K_h \nabla_h u_h^{\rm (CR,\ f)}\parallel_{0,\Om}\lesssim h^2(|u^{\rm (f,\ S)}|_{\frac{7}{2},\Om}+\kappa |\ln h|^{1/2}|u^{\rm (f,\ S)}|_{2,\infty,\Om}+|f|_{1,\infty,\Om}).
$$
\end{Th}
The superconvergence result of the CR element for the Poisson problem can also be extended to the corresponding eigenvalue problem.
\begin{Th}\label{supereig1}
Suppose that $(\lambda, u)$ is the eigenpair  of \eqref{variance}, $(\lambda^{\rm (CR, E)}_h, u^{\rm (CR, E)}_h)$ is the corresponding approximate eigenpair of  \eqref{discrete} by the CR element on an uniform triangulation $\cT_h$. Assume that $u\in H^{\frac{7}{2}}(\Om,\mathbb{R})\cap H^1_0(\Om,\mathbb{R})$, then,
$$
\parallel \nabla u-K_h \nabla_h  u^{\rm (CR, E)}_h\parallel_{0,\Om}\lesssim h^2|\ln h|^{1/2}|u|_{\frac{7}{2},\Om}.
$$
\end{Th}
\begin{proof}
Let $u^{\rm (CR,\ \lambda u)}_h\in V_h^{\rm CR}$ be the solution of the following source problem
\be\label{CRbdpro}
(\nabla_h u^{\rm (CR,\ \lambda u)}_h, \nabla_h v_h)=\lambda (u,v_h)\quad \text{ for any } v_h\in V^{\rm CR}_h.
\ee
Since $(\lambda, u)$ is the eigenpair of \eqref{variance}, it follows from Theorem \ref{superbd} that
\be\label{Rh1}
\parallel \nabla u-K_h \nabla_h  u^{\rm (CR,\ \lambda u)}_h\parallel_{0,\Om}\lesssim h^2|\ln h|^{1/2}|u|_{\frac{7}{2},\Om}.
\ee
A combination of \eqref{discrete} and \eqref{CRbdpro} yields
\be\label{CRESdiff0}
\parallel\nabla_h u^{\rm (CR, E)}_h-\nabla_h u^{\rm (CR,\ \lambda u)}_h\parallel_{0,\Om}^2=(\lambda^{\rm (CR, E)}_h u^{\rm (CR, E)}_h - \lambda u, u^{\rm (CR, E)}_h-u^{\rm (CR,\ \lambda u)}_h)
\ee
By \eqref{CR:est},
\begin{equation}\label{CRESdiff1}
\parallel \lambda^{\rm (CR, E)}_h u^{\rm (CR, E)}_h - \lambda u\parallel_{0,\Om}\leq |\lambda^{\rm (CR, E)}_h|\parallel u^{\rm (CR, E)}_h - u\parallel_{0,\Om}+ | \lambda^{\rm (CR, E)}_h -  \lambda |\lesssim h^2|u|_{2,\Om}.
\end{equation}
Thanks to \eqref{CRbdpro},
\be\label{CRESdiff2}
\parallel u^{\rm (CR, E)}_h-u^{\rm (CR,\ \lambda u)}_h\parallel_{0,\Om}\leq \parallel u^{\rm (CR, E)}_h-u\parallel_{0,\Om}+ \parallel u-u^{\rm (CR,\ \lambda u)}_h\parallel_{0,\Om}\lesssim h^2|u|_{2,\Om}.
\ee
A substitution of \eqref{CRESdiff1} and \eqref{CRESdiff2} to \eqref{CRESdiff0} leads to
\be\label{CRESdiff}
\parallel\nabla_h u^{\rm (CR, E)}_h-\nabla_h u^{\rm (CR,\ \lambda u)}_h\parallel_{0,\Om}^2\lesssim h^4|u|_{2,\Om}^2.
\ee
As pointed out in \cite{Brandts1994Superconvergence}, the gradient recovery operator $K_h$ is bounded. By \eqref{CRESdiff},
\be\label{Rh2}
\parallel K_h  \nabla_h u^{\rm (CR, E)}_h-K_h  \nabla_h u^{\rm (CR,\ \lambda u)}_h\parallel_{0,\Om}\lesssim \parallel\nabla_h u^{\rm (CR, E)}_h-\nabla_h u^{\rm (CR,\ \lambda u)}_h\parallel_{0,\Om}\lesssim h^2|u|_{2,\Om}.
\ee
It follows from \eqref{Rh1} and \eqref{Rh2} that
$$
\parallel \nabla u-K_h \nabla_h  u^{\rm (CR, E)}_h\parallel_{0,\Om} \lesssim h^2|\ln h|^{1/2}|u|_{\frac{7}{2},\Om},
$$
which completes the proof.
\end{proof}

As analyzed in \cite{Brandts1994Superconvergence}, the vector $K_h \Pi_h^{\rm RT}\sigma^{\rm (f,\ S)}$ is a higher order approximation of $\sigma^{\rm (f,\ S)}$ than $\Pi_h^{\rm RT}\sigma^{\rm (f,\ S)}$ itself. Thanks to the equivalence between the RT element and the ECR element \cite{Hu2015The}, namely,
$$
\sigma_h^{\rm (RT,\ f)}= \nabla_h u^{\rm (ECR,\ f)}_h,
$$
a similar argument proves the following superconvergence result for the ECR element.

\begin{Th}\label{supereig2}
Suppose that $(\lambda, u)$ is the eigenpair of \eqref{variance}, $(\lambda^{\rm (ECR, E)}_h, u^{\rm (ECR, E)}_h)$ is the corresponding approximate eigenpair  of  \eqref{discrete} by the ECR element on an uniform triangulation $\cT_h$. Assume that $u\in H^{\frac{7}{2}}(\Om,\mathbb{R})\cap H^1_0(\Om,\mathbb{R})$, then,
$$
\parallel \nabla u-K_h \nabla_h  u^{\rm (ECR, E)}_h\parallel_{0,\Om}\lesssim h^2|\ln h|^{1/2}|u|_{\frac{7}{2},\Om}.
$$
\end{Th}

\section{Asymptotic expansions of eigenvalues by the CR element and the ECR element}
In this section, asymptotic expansions of eigenvalues are established and employed to achieve  high accuracy extrapolation methods for both the CR element and the ECR element.

For the CR element and the ECR element, their canonical interpolations do not admit the usual superclose property with respect to the finite element solutions in the energy norm. Thus, it is difficult to establish asymptotic expansions of eigenvalues by only using the canonical interpolations. To overcome such a difficulty, a new idea is proposed to expand the errors $\nabla u- \nabla_h u_h^{\rm (CR, E)}$ and $\nabla u- \nabla_h u_h^{\rm (ECR, E)}$ in \eqref{commutId}. The key of the idea is to use the canonical interpolation operator $\Pi_h^{\rm RT}$ of the RT element, instead of $\Pi_h^{\rm CR}$ and $\Pi_h^{\rm ECR}$ defined in \eqref{crinterpolation} and \eqref{ecrinterpolation}, respectively. To this end, consider the following source problem: seeks $(\sigma^{\rm (RT,\ \lambda u)}_h,u^{\rm (RT,\ \lambda u)}_h) \in \text{RT}(\cT_h)\times U_{\text{RT}}$ such that
\begin{equation}\label{RTbdPro}
\begin{aligned}
(\sigma^{\rm (RT,\ \lambda u)}_h,\tau_h)+(u^{\rm (RT,\ \lambda u)}_h,\text{div}\tau_h)&=0&& \text{ for any }\tau_h\in \text{RT}(\cT_h),\\
(\text{div}\sigma^{\rm (RT,\ \lambda u)}_h,v_h)&=-\lambda (u,v_h)&&\text{ for any }v_h\in U_{\text{RT}}.
\end{aligned}
\end{equation}
Note that $\sigma^{\rm (RT,\ \lambda u)}_h $ is the RT element solution of $\sigma^{\rm (\lambda u,\ S)}:=\nabla u$. It follows from the theory of mixed finite element methods \cite{Douglas1985Global} that
\be\label{RT:est}
\parallel u- u^{\rm (RT,\ \lambda u)}_h\parallel_{0,\Om} + \parallel \sigma^{\rm (\lambda u,\ S)} -\sigma^{\rm (RT,\ \lambda u)}_h\parallel_{0,\Om}\lesssim h^{s}\parallel u\parallel_{1+s,\Om},
\ee
provided that $u\in H^{1+s}(\Om,\mathbb{R})\cap H^1_0(\Om,\mathbb{R})$, $\ 0< s\leq 1$.

\subsection{Taylor expansions of interpolation errors}
Denote the interpolation operators $\Pi_K^0: L^2(K)\rightarrow P_0(K)$ and $\Pi_h^0: L^2(\Om)\rightarrow U_{\rm RT}$ by
\be\label{Pi0est}
\Pi_K^0 w =\frac{1}{K}\int_K w\dx\quad \text{ and }\quad \Pi_h^0 w\big |_K =\frac{1}{K}\int_K w\dx,
\ee
respectively. On each element $K$, denote the centroid of element $K$ by $\bold{M}_K=(M_1,M_2)$ and the midpoint of edge $e_i$ by $\bold{m}_i $. Let $A_K=\sum_{i,j =1, i\neq j}^3 \big ((p_{i1}-p_{j1})^2 - (p_{i2}-p_{j2})^2\big )$, $B_K= \sum_{i=1}^3\big( 2p_{i1}p_{i2} -\sum_{j=1,j\neq i}^3 p_{i1}p_{j2}\big  )$ and $H_K=\sum_{i=1}^3 |e_i|^2 $. We introduce five short-hand notations
\begin{align*}\label{consDef}
\phi^{\rm RT}_1(\bold{x})&=(x_1-M_1, -x_2+M_2)^T,&\phi^{\rm RT}_2(\bold{x})=(x_2-M_2, x_1-M_1)^T,\\
\phi^{\rm ECR}_1(\bold{x})&=(x_1-M_1)^2-(x_2-M_2)^2,&\phi^{\rm ECR}_2(\bold{x})=(x_1-M_1)(x_2-M_2),\\
\phi^{\rm ECR}_3(\bold{x})&=2 - \frac{36}{H_K}\sum_{i=1}^2(x_i-M_i)^2.&
\end{align*}
Note that functions $\phi^{\rm RT}_1$ and $\phi^{\rm RT}_2$ belong to the compliment space of the shape function space of the RT element with respect to $\nabla P_2(K)$, and functions $\phi^{\rm ECR}_1$, $\phi^{\rm ECR}_2$ and $\phi^{\rm ECR}_3$ belong to the compliment space of the shape function space of the ECR element with respect to $P_2(K)$.

For any $w\in H^2(K,\mathbb{R})$, define the Taylor expansions of the interpolation errors $(I-\Pi^{\rm CR}_h)w $, $(I-\Pi^{\rm ECR}_h)w $ and $(I-\Pi_h^{\rm RT})\nabla w $ by
\begin{align*}
P_K^{\rm CR}(\nabla^2 w)&=\frac{\Pi_K^0 (\partial_{ x_1x_1 } w-\partial_{ x_2x_2} w)}{4}(I-\Pi_h^{\rm ECR})\phi^{\rm ECR}_1 + \Pi_K^0 \partial_{ x_1x_2} w(I-\Pi_h^{\rm ECR})\phi^{\rm ECR}_2\\
&+\big(-\frac{A_K+H_K}{144}\Pi_K^0 \partial_{ x_1x_1 } w-\frac{H_K-A_K}{144}\Pi_K^0 \partial_{ x_2x_2 } w-\frac{B_K}{36}\Pi_K^0 \partial_{ x_1x_2}w\big)\phi^{\rm ECR}_3,\\
P_K^{\rm ECR}(\nabla^2 w)&=\frac{\Pi_K^0 (\partial_{ x_1x_1 } w-\partial_{ x_2x_2} w)}{4}(I-\Pi_h^{\rm ECR})\phi^{\rm ECR}_1 + \Pi_K^0 \partial_{ x_1x_2} w(I-\Pi_h^{\rm ECR})\phi^{\rm ECR}_2,\\
P_K^{\rm RT}(\nabla^2 w)&=\frac{\Pi^0_K (\partial_{x_1x_1}w-\partial_{x_2x_2}w)}{2}(I-\Pi_h^{\rm RT})\phi^{\rm RT}_1 + \Pi^0_K \partial_{x_1x_2}w (I-\Pi_h^{\rm RT})\phi^{\rm RT}_2.
\end{align*}
\begin{Lm}\label{Lm:cr}
For any quadratic function $w\in P_2(K)$,
\begin{equation}\label{identity:cr}
(I-\Pi^{\rm CR}_h)w\big|_K=P_K^{\rm CR}(\nabla^2 w),
\end{equation}
\begin{equation}\label{identity:ecr}
(I-\Pi^{\rm ECR}_h)w\big|_K=P_K^{\rm ECR}(\nabla^2 w),
\end{equation}
\begin{equation}\label{identity:RT}
(I-\Pi^{\rm RT}_h)\nabla w\big |_K=P_K^{\rm RT}(\nabla^2 w),
\end{equation}
\begin{equation}\label{identity:RT2}
\begin{split}
\parallel P_K^{\rm RT}(\nabla^2 w)\parallel_{0,K}^2=&\frac{1}{4}c_{11}^{\rm RT}|_K(\partial_{x_1x_1}w-\partial_{x_2x_2}w)^2 +  c_{22}^{\rm RT}|_K(\partial_{x_1x_2}w)^2 \\
&+ c_{12}^{\rm RT}|_K(\partial_{x_1x_1}w\partial_{x_1x_2}w-\partial_{x_2x_2}w\partial_{x_1x_2}w),
\end{split}
\end{equation}
where $c_{ij}^{\rm RT}|_K=\int_K \big((I-\Pi_h^{\rm RT})\phi^{\rm RT}_i\big )^T(I-\Pi_h^{\rm RT})\phi^{\rm RT}_j\dx$, $1\leq i\leq 6, j= 1, 2$ are constant.
\end{Lm}
\begin{proof}
Note that $\phi^{\rm ECR}_1$, $\phi^{\rm ECR}_2$ and $\phi^{\rm ECR}_3$ are linearly independent, and
$$
P_2(K)= P_1(K)+\rm span\{\phi^{\rm ECR}_i, 1\leq i\leq 3\}.
$$
Since
$$
\Pi_h^{\rm CR}(I-\Pi_h^{\rm ECR})\phi^{\rm ECR}_i=0,\quad \Pi_h^{\rm CR}\phi^{\rm ECR}_3=0\quad\text{ for any }i=1, 2,
$$
The interpolation error $(I-\Pi^{\rm CR}_h)w$ can be expressed in the following form:
$$
(I-\Pi^{\rm CR}_h)w=c_1(I-\Pi_h^{\rm ECR})\phi^{\rm ECR}_1 + c_2(I-\Pi_h^{\rm ECR})\phi^{\rm ECR}_2 +c_3\phi^{\rm ECR}_3.
$$
The coefficients $c_i$ can be determined by taking second order derivatives on both sides of the above identity. It leads to

$$
c_1=\frac{1}{4} (\partial_{x_1x_1}w-\partial_{x_2x_2}w), \text{\quad }c_2=\partial_{x_1x_2}w,
$$
and
$$
c_3=-\frac{A_K+H_K}{144}\partial_{x_1x_1}w-\frac{H_K-A_K}{144}\partial_{x_2x_2}w-\frac{B_K}{36}\partial_{x_1x_2}w\emph{},
$$
which proves \eqref{identity:cr}. A similar procedure verifies the expansions \eqref{identity:ecr}, \eqref{identity:RT} and \eqref{identity:RT2}, which completes the proof.
\qed
\end{proof}
Refine a triangulation $\cT_{2h}$ into a half-sized triangulation uniformly to obtain $\cT_h$, namely, for any element $K_{2h}\in \cT_{2h}$, $K_{2h}=\cup_{l=1}^4 K_{h}^l$ where $K_{h}^l\in \cT_h, 1\leq l\leq 4$.
\begin{Lm}\label{Lm:4timesRT}
For any $w\in P_2(K_{2h})$, it holds that
\begin{align}
\parallel P_{K_{2h}}^{\rm RT}(\nabla^2 w)\parallel_{0,K_{2h}}^2 &= 4\sum_{l=1}^4\parallel P_{K_{h}^l}^{\rm RT}(\nabla^2 w)\parallel_{0,K_{h}^l}^2,\label{RT4time}\\
\int_{K_{2h}}P^{\rm ECR}_{K_{2h}}(\nabla^2 w)\dx &=4\sum_{l=1}^4\int_{K_{h}^l} P^{\rm ECR}_{K_{h}^l}(\nabla^2 w)\dx.\label{nablau4time}
\end{align}
\end{Lm}
\begin{proof}
In order to prove \eqref{RT4time}, it only needs to prove that
$$
\parallel P_{K_{2h}}^{\rm RT}(\nabla^2 w)\parallel_{0,K_{2h}}^2 = 16\parallel P_{K_{h}^l}^{\rm RT}(\nabla^2 w)\parallel_{0,K_{h}^l}^2,\quad 1\leq l\leq 4.
$$
For simplicity, only the case $l=1$ is considered here. Let $\bold{M}^{2h}=(M^{2h}_1, M^{2h}_2)$, $\{\bold{p}^{2h}_i\}_{i=1}^3$ and $\{e_i^{2h}\}_{i=1}^3$ be the centroid,  vertices and edges of element $K_{2h}$, respectively, and $\bold{M}^{h}=(M^{h}_1, M^{h}_2)$, $\{\bold{p}^{h}_i\}_{i=1}^3$ and $\{e_i^{h}\}_{i=1}^3$ be those of element $K_{h}^1$, respectively. For edge $e_i^{2h}$, denote the midpoint, the unit outward normal vector and the perpendicular height by $\bold{m}_i^{2h}$, $\bold{n}_i^{2h}$ and $d_i^{2h}$, respectively. And denote those of edge $e_i^h$ by $\bold{m}_i^{h}$, $\bold{n}_i^h$ and $d_i^{h}$, respectively. Let
\begin{equation*}
\begin{split}
\varphi^{\rm RT}_1(\bold{x}) = (x_1-M^{2h}_1, M^{2h}_2-x_2)^T,&\quad \varphi^{\rm RT}_2(\bold{x}) = (x_2-M^{2h}_2, x_1-M^{2h}_1)^T,\\
\psi^{\rm RT}_1(\bold{x}) = (x_1-M^{h}_1, M^{h}_2-x_2)^T,&\quad \psi^{\rm RT}_2(\bold{x}) = (x_2-M^{h}_2, x_1-M^{h}_1)^T.
\end{split}
\end{equation*}
Note that $\{\bold{x}-\bold{p}^{2h}_i\}_{i=1}^3$ and $\{\bold{x}-\bold{p}^{h}_i\}_{i=1}^3$ are the basis functions of the RT element on elements $K_{2h}$ and $K_{h}^1$, respectively. For $i=1, 2$,
\begin{equation*}
\begin{split}
(I-\Pi_{2h}^{\rm RT})\varphi^{\rm RT}_i(\bold{x})&=\varphi^{\rm RT}_i(\bold{x})-\sum_{j=1}^3a_{ij}^{2h}(\bold{x}-\bold{p}^{2h}_j),\\
(I-\Pi_h^{\rm RT})\psi^{\rm RT}_i(\bold{x})&=\psi^{\rm RT}_i(\bold{x})-\sum_{j=1}^3a_{ij}^h(\bold{x}-\bold{p}^h_j),
\end{split}
\end{equation*}
where $a_{ij}^{2h}=\frac{1}{d_j^{2h}}(\varphi^{\rm RT}_i(\bold{m}_j^{2h}))^T\bold{n}_j^{2h}$ and $a_{ij}^{h}=\frac{1}{d_j^{h}}(\varphi^{\rm RT}_i(\bold{m}_j^{h}))^T\bold{n}_j^{h}$. For each $1\leq i\leq 2$, $1\leq j\leq 3$, $\bold{m}_j^{2h}-\bold{M}^{2h}=2(\bold{m}_j^{h}-\bold{M}^{h})$, $d_j^{2h}=2d_j^{h}$ and $\bold{n}_j^{2h}=\bold{n}_j^{h}$. Thus, $a_{ij}^{2h}=a_{ij}^{h}$. By \eqref{identity:RT2}, it only remains to prove that
\begin{equation}\label{RT4relation}
c_{ij}^{\rm RT}|_{K_{2h}}=16c_{ij}^{\rm RT}|_{K_{h}^1}.
\end{equation}
Since $\int_{K_{2h}} (\bold{x}-\bold{M}^{2h})\dx=\bold{0}$ and $\int_{K_{2h}} (x_i - M^{2h}_i)(x_j - M^{2h}_j)\dx =16\int_{K_{h}^1} (x_i - M^{h}_i)(x_j - M^{h}_j)\dx$,
it holds that
\begin{equation*}\label{RT4expand}
\begin{split}
c_{11}^{\rm RT}|_{K_{2h}}=&\int_{K_{2h}} \sum_{i=1}^2(x_i - M^{2h}_i)^2 - 2\sum_{i=1}^3a_{1i}(x_1-M^{2h}_1, M^{2h}_2-x_2)(\bold{x}-\bold{M}^{2h})\\
& +\sum_{i,j=1}^3 a_{1i}a_{1j}\big ((\bold{x}-\bold{M}^{2h})^T (\bold{x}-\bold{M}^{2h}) + (\bold{M}^{2h}-\bold{p}^{2h}_i)^T(\bold{M}^{2h}-\bold{p}^{2h}_j)\big )\dx\\
=&16c_{11}^{\rm RT}|_{K_{h}^1}.
\end{split}
\end{equation*}
Similarly, \eqref{RT4relation} holds for $1\leq i, j\leq 2$, which completes the proof for \eqref{RT4time}.

A similar procedure proves \eqref{nablau4time}, which completes the proof.
\end{proof}
\subsection{Asymptotic expansions of eigenvalues by the ECR element}
Consider the source problem: seeks $u^{\rm (ECR,\ \lambda u)}_h\in V^{\rm ECR}_h$ such that
\be\label{ECRbdpro}
(\nabla u^{\rm (ECR,\ \lambda u)}_h,\nabla v_h)=(\lambda u ,v_h)\quad\text{ for any } v_h\in V^{\rm ECR}_h.
\ee
By \eqref{ECR:est} and a similar procedure for \eqref{CRESdiff},
\begin{equation*}\label{ecrdiff}
\parallel \nabla_h( u^{\rm (ECR,\ \lambda u)}_h - u^{\rm (ECR, E)}_h)\parallel_{0,\Om}\lesssim h^{2}|u|_{2,\Om},
\end{equation*}
provided that $u\in H^2(\Om, \mathbb{R})\cap H^1_0(\Om, \mathbb{R})$.

The following equivalence between the ECR element and the RT element \cite{Hu2015The} is crucial for expansions of eigenvalues by the ECR element
$$
\sigma^{\rm (RT,\ \lambda u)}_h= \nabla_h u^{\rm (ECR,\ \lambda u)}_h.
$$
As a consequence,
\begin{equation}\label{ECR3}
\parallel \sigma^{\rm (RT,\ \lambda u)}_h - \nabla_h u^{\rm (ECR, E)}_h\parallel_{0,\Om}\lesssim h^{2}|u|_{2,\Om},
\end{equation}
provided that $u\in  H^2(\Om, \mathbb{R})\cap H^1_0(\Om, \mathbb{R})$.

\begin{Lm}\label{Th:extraECR}
Suppose that $(\lambda, u)$ is the eigenpair of \eqref{variance} with $u \in H^{\frac{7}{2}}(\Om,\mathbb{R})\cap H^1_0(\Om,\mathbb{R})$, and $(\lambda^{\rm (ECR, E)}_h, u^{\rm (ECR, E)}_h)$ is the corresponding approximate eigenpair of \eqref{discrete} by the ECR element on an uniform triangulation $\cT_h$. It holds that
$$
\big |\parallel \nabla u-\nabla_h u^{\rm (ECR, E)}_h\parallel_0^2-\parallel \nabla u-\Pi^{\rm RT}_h\nabla u \parallel_0^2\big |\lesssim h^{3}|\ln h|^{1/2}|u|_{\frac{7}{2},\Om}^2.
$$
\end{Lm}
\begin{proof}
By the interpolation $\Pi^{\rm RT}_h\nabla u $, the solution $\sigma^{\rm (RT,\ \lambda u)}_h $ of the source problem \eqref{RTbdPro} by the RT element, the error of the approximate eigenfunction  in energy norm can be decomposed as
\begin{equation}\label{ECRtotal1}
\begin{split}
\parallel \nabla u-\nabla_h u^{\rm (ECR, E)}_h\parallel_{0,\Om}^2=&\parallel \nabla u-\Pi^{\rm RT}_h\nabla u \parallel_{0,\Om}^2 + \parallel \Pi^{\rm RT}_h\nabla u-\sigma^{\rm (RT,\ \lambda u)}_h \parallel_{0,\Om}^2\\
 + &\parallel \sigma^{\rm (RT,\ \lambda u)}_h- \nabla_h u^{\rm (ECR, E)}_h \parallel_{0,\Om}^2+2(\nabla u-\Pi^{\rm RT}_h\nabla u, \Pi^{\rm RT}_h\nabla u-\sigma^{\rm (RT,\ \lambda u)}_h) \\
+& 2(\nabla u-\Pi^{\rm RT}_h\nabla u, \sigma^{\rm (RT,\ \lambda u)}_h-\nabla_h u^{\rm (ECR, E)}_h)\\
+ &2(\Pi^{\rm RT}_h\nabla u-\sigma^{\rm (RT,\ \lambda u)}_h,\sigma^{\rm (RT,\ \lambda u)}_h- \nabla_h u^{\rm (ECR, E)}_h).
\end{split}
\end{equation}
Since $\sigma^{\rm (RT,\ \lambda u)}_h$ is the RT element solution of $\sigma^{\rm (\lambda u,\ S)}=\nabla u$, the superconvergence of the RT element in Theorem \ref{Lm:bdRT} reads
\be\label{ECR1}
\parallel \Pi^{\rm RT}_h\nabla u-\sigma^{\rm (RT,\ \lambda u)}_h \parallel_{0,\Om}\lesssim h^2|\ln h|^{1/2}|u|_{\frac{7}{2},\Om},
\ee
which leads to
\be\label{ECR2}
\big |(\nabla u-\Pi^{\rm RT}_h\nabla u, \Pi^{\rm RT}_h\nabla u-\sigma^{\rm (RT,\ \lambda u)}_h)\big |\lesssim h^{3}|\ln h|^{1/2}|u|_{\frac{7}{2},\Om}|u |_{2,\Om}.
\ee
The superconvergence result \eqref{ECR3} implies
\be\label{ECR4}
\big | (\nabla u-\Pi^{\rm RT}_h\nabla u, \sigma^{\rm (RT,\ \lambda u)}_h-\nabla_h u^{\rm (ECR, E)}_h)\big |\lesssim h^{3}|u|_{2,\Om}^2,
\ee
It follows from \eqref{ECR3} and \eqref{ECR1} that
\be\label{ECR5}
\big |(\Pi^{\rm RT}_h\nabla u-\sigma^{\rm (RT,\ \lambda u)}_h,\sigma^{\rm (RT,\ \lambda u)}_h- \nabla_h u^{\rm (ECR, E)}_h)\big |\lesssim h^{4}|\ln h|^{1/2}|u|_{\frac{7}{2},\Om}|u|_{2,\Om}.
\ee
A substitution of \eqref{ECR3}, \eqref{ECR1}, \eqref{ECR2}, \eqref{ECR4} and \eqref{ECR5} into \eqref{ECRtotal1} concludes
$$
\big |\parallel \nabla u-\nabla_h u^{\rm (ECR, E)}_h\parallel_{0,\Om}^2-\parallel \nabla u-\Pi^{\rm RT}_h\nabla u \parallel_{0,\Om}^2\big |\lesssim h^{3}|\ln h|^{1/2}|u|_{\frac{7}{2},\Om}^2,
$$
which completes the proof.
\end{proof}
In the following theorem, asymptotic expansions of eigenvalues by the ECR element are established and employed to prove that the accuracy of eigenvalues is improved from $O(h^2)$ to $O(h^3)$ by extrapolation methods.

\begin{Th}\label{ECR:extra}
Suppose that $(\lambda, u)$ is the eigenpair of \eqref{variance} with $u\in H^{\frac{7}{2}}(\Om,\mathbb{R})\cap H^1_0(\Om,\mathbb{R})$, and $(\lambda^{\rm (ECR, E)}_h, u^{\rm (ECR, E)}_h)$ is the corresponding approximate eigenpair of \eqref{discrete} by the ECR element on an uniform triangulation $\cT_h$. It holds that
$$
\big |\lambda-\lambda^{\rm ECR}_{\rm EXP}\big |\lesssim h^3|\ln h|^{1/2}|u|_{\frac{7}{2},\Om}^2,
$$
where the extrapolation eigenvalue $\lambda^{\rm ECR}_{\rm EXP}=\frac{4\lambda^{\rm (ECR, E)}_h-\lambda^{\rm (ECR, E)}_{2h}}{3}$.
\end{Th}
\begin{proof}
A similar identity of \eqref{commutId} holds for the ECR element, namely,
\begin{equation*}
\begin{split}
\lambda-\lambda^{\rm (ECR, E)}_h=&\parallel \nabla_h (u-u^{\rm (ECR, E)}_h)\parallel_{0,\Om}^2-2\lambda^{\rm (ECR, E)}_h(u-\Pi^{\rm ECR}_h u,u^{\rm (ECR, E)}_h)\\
&-\lambda^{\rm (ECR, E)}_h\parallel u-u^{\rm (ECR, E)}_h\parallel_{0,\Om}^2.
\end{split}
\end{equation*}
It follows from Lemma \ref{Th:extraECR} that
\begin{equation*}
\setlength{\abovedisplayskip}{6pt}
\setlength{\belowdisplayskip}{6pt}
\begin{split}
\lambda-\lambda^{\rm (ECR, E)}_h&=\parallel \nabla u-\Pi^{\rm RT}_h\nabla u \parallel_{0,\Om}^2-2\lambda(u-\Pi^{\rm ECR}_h u,\Pi_h^0 u)-2\lambda(u-\Pi^{\rm ECR}_h u,u-\Pi_h^0 u)\\
&-2(\lambda^{\rm (ECR, E)}_h-\lambda)(u-\Pi^{\rm ECR}_h u,u)-2\lambda^{\rm (ECR, E)}_h(u-\Pi^{\rm ECR}_h u,u^{\rm (ECR, E)}_h-u)\\
&+O(h^3|\ln h|^{1/2}|u|_{\frac{7}{2},\Om}^2).
\end{split}
\end{equation*}
By \eqref{ECR:est},
$$
|\lambda(u-\Pi^{\rm ECR}_h u,u-\Pi_h^0 u)|\lesssim h^3|u|_{2,\Om} |u|_{1,\Om},
$$
$$
\big |(u-\Pi^{\rm ECR}_h u,u^{\rm (ECR, E)}_h-u)\big |\lesssim \parallel u-\Pi^{\rm ECR}_h u \parallel_{0,\Om}\parallel u^{\rm (ECR, E)}_h-u\parallel_{0,\Om}\lesssim h^4|u|_{2,\Om}^2,
$$
$$
\big |(\lambda_h^{\rm (ECR, E)}-\lambda)(u -\Pi^{\rm ECR}_h u,u)\big |\lesssim |\lambda_h^{\rm (ECR, E)}-\lambda| \parallel u-\Pi^{\rm ECR}_h u \parallel_{0,\Om}\lesssim h^4|u|_{2,\Om}^2.
$$
As a consequence,
\begin{equation*}
\big |\lambda-\lambda^{\rm (ECR, E)}_h-\big (\parallel \nabla u-\Pi^{\rm RT}_h\nabla u \parallel_{0,\Om}^2-2\lambda (u -\Pi^{\rm ECR}_h u ,\Pi_h^0 u )\big )\big |\lesssim h^3|\ln h|^{1/2}|u |_{\frac{7}{2},\Om}^2.
\end{equation*}
Due to the Bramble-Hilbert lemma,  \eqref{identity:ecr} and \eqref{identity:RT},
\begin{equation}\label{ECRexpan}
\lambda^{\rm (ECR, E)}_h=\lambda -\sum_{K\in\cT_h}\big ( \parallel P_K^{\rm RT}(\nabla^2 u )\parallel_{0,K}^2 -2\lambda \Pi_K^0 u \int_K P^{\rm ECR}_K(\nabla^2 u )\dx\big )+O(h^3|\ln h|^{1/2}|u |_{\frac{7}{2},\Om}^2).
\end{equation}
Since
$
\parallel u -\Pi_K^0 u \parallel_{0,K}\lesssim h|u |_{1,K},
$
a combination of \eqref{ECRexpan} and Lemma \ref{Lm:4timesRT} concludes
$$
\big |\lambda -\frac{4\lambda^{\rm (ECR, E)}_h-\lambda^{\rm (ECR, E)}_{2h}}{3}\big |\lesssim h^3|\ln h|^{1/2}|u |_{\frac{7}{2},\Om}^2,
$$
which completes the proof.
\end{proof}

\subsection{Asymptotic expansions of eigenvalues by the CR element}
Let $u^{\rm (CR,\ \lambda\Pi_h^0 u)}_h\in V^{\rm CR}_h$ be the solution of the following source problem
\be\label{CRbdPipro}
(\nabla u^{\rm (CR,\ \lambda\Pi_h^0 u)}_h,\nabla v_h)=\lambda (\Pi_h^0 u ,v_h)\quad \text{ for any } v_h\in V^{\rm CR}_h,
\ee
and $u^{\rm (CR,\ \lambda u)}_h\in V^{\rm CR}_h$ be the solution of the source problem \eqref{CRbdpro}. The theory of nonconforming approximation \cite{ShiWangBook} gives the following lemma.
\begin{Lm}\label{Lm:crdiff}
Suppose that $(\lambda, u)$ is the eigenpair of \eqref{variance} with $u\in H^{2}(\Om,\mathbb{R})\cap H^1_0(\Om,\mathbb{R})$, $(\lambda^{\rm (CR, E)}_h, u^{\rm (CR, E)}_h)$ is the discrete eigenpair of  \eqref{discrete} in $V_h^{\rm CR}$ and $u^{\rm (CR,\ \lambda\Pi_h^0 u)}_h$ is the solution of \eqref{CRbdPipro} by the CR element, respectively. It holds that
\begin{equation*}\label{crdiff}
\parallel \nabla_h( u^{\rm (CR,\ \lambda\Pi_h^0 u)}_h - u^{\rm (CR, E)}_h)\parallel_{0,\Om}\lesssim h^{2}|u |_{2,\Om}.
\end{equation*}
\end{Lm}
\begin{proof}
Due to \eqref{CRbdpro} and \eqref{CRbdPipro}, for any $v_h\in V_h^{\rm CR}$,
\begin{equation*}
\big (\nabla_h( u^{\rm (CR,\ \lambda u)}_h - u^{\rm (CR,\ \lambda\Pi_h^0 u)}_h), \nabla v_h\big )=\big (\lambda  u -\lambda \Pi_h^0 u  ,v_h\big )=\lambda \big((I-\Pi_h^0)u ,(I-\Pi_h^0)v_h\big ).
\end{equation*}
It follows from \eqref{Pi0est} that
\begin{equation*}
\begin{split}
\big |\big (\nabla_h( u^{\rm (CR,\ \lambda\Pi_h^0 u)}_h - u^{\rm (CR,\ \lambda u)}_h), \nabla v_h\big )\big |&\lesssim h^2|u |_{1,\Om}|v_h|_{1,h}\leq C h^4|u |_{1,\Om}^2+\frac{1}{2}|v_h|_{1,h}^2.
\end{split}
\end{equation*}
With $v_h=u^{\rm (CR,\ \lambda\Pi_h^0 u)}_h - u^{\rm (CR,\ \lambda u)}_h$ in the above inequality,
\be\label{CRdiffest:1}
\parallel \nabla_h( u^{\rm (CR,\ \lambda\Pi_h^0 u)}_h - u^{\rm (CR,\ \lambda u)}_h)\parallel_{0,\Om}\lesssim h^2|u |_{1,\Om}.
\ee
Thanks to \eqref{CRESdiff},
\be\label{CRdiffest:2}
\parallel \nabla_h (u^{\rm (CR,\ \lambda u)}_h - u^{\rm (CR, E)}_h)\parallel_{0,\Om}\lesssim h^2|u |_{2,\Om}.
\ee
A combination of \eqref{CRdiffest:1} and \eqref{CRdiffest:2} concludes
$$
\parallel \nabla_h( u^{\rm (CR,\ \lambda\Pi_h^0 u)}_h - u^{\rm (CR, E)}_h)\parallel_{0,\Om}\lesssim h^{2}|u |_{2,\Om},
$$
which completes the proof.
\end{proof}
The following special relation between the CR element and the RT element was analyzed in \cite{Marini1985An}
\begin{equation}\label{CRRT}
\sigma^{\rm (RT,\ \lambda u)}_h\big|_K=\nabla_h u^{\rm (CR,\ \lambda\Pi_h^0 u)}_h\big|_K-\frac{\lambda  \Pi_K^0 u }{2}(\bold{x}-\bold{M}_K),\quad \bold{x}\in K, \text{ for any }K\in\cT_h.
\end{equation}
It plays an important role in the analysis of asymptotic expansions of eigenvalues by the CR element. Let $(\sigma^{\rm (RT,\ \lambda u)}_h, u^{\rm (RT,\ \lambda u)}_h)\in \rm RT(\cT_h)\times U_{\rm RT} $ be the solution of \eqref{RTbdPro}, and $u^{\rm (CR,\ \lambda\Pi_h^0 u)}_h$ be the solution of \eqref{CRbdPipro}. By \eqref{CRRT}, a direct computation yields
\begin{equation}\label{CRRTL2}
\parallel \sigma^{\rm (RT,\ \lambda u)}_h- \nabla_h u^{\rm (CR,\ \lambda\Pi_h^0 u)}_h \parallel_{0,\Om}= \frac{\lambda }{12}\big (\sum_{K\in \cT_h}H_K\parallel \Pi_K^0 u \parallel_{0,K}^2\big )^{1/2},
\end{equation}
\begin{equation}\label{CRRTconstant}
\int_K (\sigma^{\rm (RT,\ \lambda u)}_h- \nabla_h u^{\rm (CR,\ \lambda\Pi_h^0 u)}_h)\dx =0\quad \text{ for any } K\in \cT_h.
\end{equation}

\begin{Lm}\label{Th:extraCR}
Suppose that $(\lambda , u )$ is the eigenpair of \eqref{variance} with $u \in H^{\frac{7}{2}}(\Om,\mathbb{R})\cap H^1_0(\Om,\mathbb{R})$, and $(\lambda^{\rm (CR, E)}_h, u^{\rm (CR, E)}_h)$ is the corresponding approximate eigenpair  of \eqref{discrete} by the CR element on an uniform triangulation $\cT_h$. Then,
\begin{equation*}
\begin{split}
\parallel \nabla u -\nabla_h u^{\rm (CR, E)}_h\parallel_{0,\Om}^2=&\sum_{K\in\cT_h}\parallel P_K^{\rm RT}(\nabla^2 u ) \parallel_{0,K}^2 + \frac{\lambda^2}{144}\sum_{K\in \cT_h}\parallel u \parallel_{0,K}^2|\partial K|^2\\
&-\lambda \sum_{K\in\cT_h}\int_K (\bold{x}-\bold{M}_K)P_K^{\rm RT}(\nabla^2 u )u \dx +O(h^{3}|\ln h|^{1/2}|u |_{\frac{7}{2},\Om}^2 ).
\end{split}
\end{equation*}
\end{Lm}
\begin{proof}
A similar procedure of \eqref{ECRtotal1} yields
\begin{equation}\label{CRtotal0}
\begin{split}
\parallel \nabla u -\nabla_h u^{\rm (CR, E)}_h\parallel_{0,\Om}^2&=\parallel \nabla u -\Pi^{\rm RT}_h\nabla u  \parallel_{0,\Om}^2 + \parallel \Pi^{\rm RT}_h\nabla u -\sigma^{\rm (RT,\ \lambda u)}_h \parallel_{0,\Om}^2 \\
+& \parallel \sigma^{\rm (RT,\ \lambda u)}_h- \nabla_h u^{\rm (CR,\ \lambda\Pi_h^0 u)}_h\parallel_{0,\Om}^2+\parallel \nabla_h u^{\rm (CR,\ \lambda\Pi_h^0 u)}_h- \nabla_h u^{\rm (CR, E)}_h\parallel_{0,\Om}^2\\
+& 2(\nabla u -\Pi^{\rm RT}_h\nabla u , \Pi^{\rm RT}_h\nabla u -\sigma^{\rm (RT,\ \lambda u)}_h)\\
+&2(\nabla u -\Pi^{\rm RT}_h\nabla u , \sigma^{\rm (RT,\ \lambda u)}_h-\nabla_h u^{\rm (CR,\ \lambda\Pi_h^0 u)}_h)\\
+& 2(\nabla u -\Pi^{\rm RT}_h \nabla u , \nabla_h u^{\rm (CR,\ \lambda\Pi_h^0 u)}_h-\nabla_h u^{\rm (CR, E)}_h)\\
+&2(\Pi^{\rm RT}_h\nabla u -\sigma^{\rm (RT,\ \lambda u)}_h,\sigma^{\rm (RT,\ \lambda u)}_h- \nabla_h u^{\rm (CR, E)}_h)\\
+&2(\sigma^{\rm (RT,\ \lambda u)}_h- \nabla_h u^{\rm (CR,\ \lambda\Pi_h^0 u)}_h, \nabla_h u^{\rm (CR,\ \lambda\Pi_h^0 u)}_h-\nabla_h u^{\rm (CR, E)}_h).
\end{split}
\end{equation}
It follows from the error estimates in \eqref{CR:est} and \eqref{RT:est} for the solution  by the CR element and the RT element that
\begin{equation}\label{oneorder}
\parallel \sigma^{\rm (RT,\ \lambda u)}_h- \nabla_h u^{\rm (CR, E)}_h \parallel_{0,\Om}\lesssim h|u |_{2,\Om}.
\end{equation}
Since $\sigma^{\rm (RT,\ \lambda u)}_h$ is the RT element solution of $\sigma^{\rm (\lambda u,\ S)}=\nabla u $, the superconvergence of the RT element in Theorem \ref{Lm:bdRT} reads
\begin{equation}\label{CR1}
\parallel \Pi^{\rm RT}_h\nabla u -\sigma^{\rm (RT,\ \lambda u)}_h \parallel_{0,\Om}\lesssim h^2|\ln h|^{1/2}|u |_{\frac{7}{2},\Om},
\end{equation}
which leads to
\begin{equation}\label{CR2}
\big |(\nabla u -\Pi^{\rm RT}_h\nabla u , \Pi^{\rm RT}_h\nabla u -\sigma^{\rm (RT,\ \lambda u)}_h)\big |\lesssim h^{3}|\ln h|^{1/2}|u |_{\frac{7}{2},\Om}|u |_{2,\Om},
\end{equation}
\begin{equation}\label{CR3}
\big |(\Pi^{\rm RT}_h\nabla u -\sigma^{\rm (RT,\ \lambda u)}_h,\sigma^{\rm (RT,\ \lambda u)}_h- \nabla_h u^{\rm (CR, E)}_h)\big |\lesssim h^{3}|\ln h|^{1/2}|u |_{\frac{7}{2},\Om}|u |_{2,\Om}.
\end{equation}
For the difference between the solution of the eigenvalue problem \eqref{discrete} and the source problem \eqref{CRbdPipro} by the CR element, it follows from Lemma \ref{Lm:crdiff} that
\begin{equation}\label{CR4}
\parallel \nabla_h u^{\rm (CR,\ \lambda\Pi_h^0 u)}_h- \nabla_h u^{\rm (CR, E)}_h\parallel_{0,\Om}\lesssim h^2|u |_{2,\Om}.
\end{equation}
This superconvergence result leads to
\begin{equation}\label{CR5}
\big | (\nabla u -\Pi^{\rm RT}_h\nabla u , \nabla_h u^{\rm (CR,\ \lambda\Pi_h^0 u)}_h -\nabla_h u^{\rm (CR, E)}_h)\big |\lesssim h^{3}|u |_{2,\Om}^2,
\end{equation}
\begin{equation}\label{CR6}
\big |(\sigma^{\rm (RT,\ \lambda u)}_h- \nabla_h u^{\rm (CR,\ \lambda\Pi_h^0 u)}_h, \nabla_h u^{\rm (CR,\ \lambda\Pi_h^0 u)}_h-\nabla_h u^{\rm (CR, E)}_h)\big |\lesssim h^3|u |_{2,\Om}^2.
\end{equation}
A substitution of \eqref{CR1}, \eqref{CR2}, \eqref{CR3}, \eqref{CR4}, \eqref{CR5}, \eqref{CR6} into \eqref{CRtotal0} yields
\begin{equation}\label{CRtotal}
\begin{split}
\parallel \nabla u -\nabla_h u^{\rm (CR, E)}_h\parallel_{0,\Om}^2=&\parallel \nabla u -\Pi^{\rm RT}_h\nabla u  \parallel_{0,\Om}^2 +\parallel \sigma^{\rm (RT,\ \lambda u)}_h- \nabla_h u^{\rm (CR,\ \lambda\Pi_h^0 u)}_h\parallel_{0,\Om}^2\\
&+2(\nabla u -\Pi^{\rm RT}_h\nabla u , \sigma^{\rm (RT,\ \lambda u)}_h-\nabla_h u^{\rm (CR,\ \lambda\Pi_h^0 u)}_h)+O(h^{3}|\ln h|^{1/2}|u |_{\frac{7}{2},\Om}^2 ).
\end{split}
\end{equation}
By the Taylor expansion of the interpolation error in \eqref{identity:RT} and the Bramble-Hilbert lemma,
\begin{equation}\label{CR:5}
\parallel \nabla u -\Pi^{\rm RT}_h\nabla u  \parallel_{0,\Om}^2= \sum_{K\in\cT_h}\parallel P_K^{\rm RT}(\nabla^2 u ) \parallel_{0,K}^2+O(h^3|u |_{3,\Om}^2 ).
\end{equation}
Due to the special relation between the CR element and the RT element in \eqref{CRRT}, a combination of \eqref{Pi0est} and \eqref{CRRTL2} yields
\begin{equation}\label{CR:6}
\parallel \sigma^{\rm (RT,\ \lambda u)}_h- \nabla_h u^{\rm (CR,\ \lambda\Pi_h^0 u)}_h\parallel_{0,\Om}^2=\frac{\lambda^2}{144}\sum_{K\in \cT_h}H_K\parallel u \parallel_{0,K}^2 +O(h^4|u |_{1,\Om}^2),
\end{equation}
with $H_K=\sum_{i=1}^3|e_i|^2$. By the Bramble-Hilbert lemma and \eqref{identity:RT}, \eqref{CRRT},
\begin{equation}\label{CR:7}
(\nabla u -\Pi^{\rm RT}_h\nabla u , \sigma^{\rm (RT,\ \lambda u)}_h-\nabla_h u^{\rm (CR,\ \lambda\Pi_h^0 u)}_h)=-\frac{\lambda }{2}\sum_{K\in\cT_h}\int_K (\bold{x}-\bold{M}_K)P_K^{\rm RT}(\nabla^2 u )u \dx+O(h^3|u |_{3,\Om}^2 ).
\end{equation}
A substitution of \eqref{CR:5}, \eqref{CR:6}, \eqref{CR:7} into \eqref{CRtotal} concludes
\begin{equation*}
\begin{split}
\parallel \nabla u -\nabla_h u^{\rm (CR, E)}_h\parallel_{0,\Om}^2=&\sum_{K\in\cT_h}\parallel P_K^{\rm RT}(\nabla^2 u ) \parallel_{0,K}^2 + \frac{\lambda^2}{144}\sum_{K\in \cT_h}H_K\parallel u \parallel_{0,K}^2\\
&-\lambda \sum_{K\in\cT_h}\int_K (\bold{x}-\bold{M}_K)P_K^{\rm RT}(\nabla^2 u )u \dx +O(h^{3}|\ln h|^{1/2}|u |_{\frac{7}{2},\Om}^2 ),
\end{split}
\end{equation*}
which completes the proof.
\end{proof}
By a similar proof for Theorem \ref{ECR:extra}, asymptotic expansions of eigenvalues by the CR element are established in the following theorem.
\begin{Th}
Suppose that $(\lambda , u )$ is the eigenpair of \eqref{variance} with $u \in H^{\frac{7}{2}}(\Om,\mathbb{R})\cap H^1_0(\Om,\mathbb{R})$, and $(\lambda^{\rm (CR, E)}_h, u^{\rm (CR, E)}_h)$ is the corresponding approximate eigenpair of \eqref{discrete} by the CR element on an uniform triangulation $\cT_h$. It holds that
\begin{equation*}
\begin{split}
\lambda^{\rm (CR, E)}_h=&\lambda -\sum_{K\in\cT_h} \big (\parallel P_K^{\rm RT}(\nabla^2 u )\parallel_{0,K}^2 -2\lambda ( P^{\rm CR}_K(\nabla^2 u ),u )\big )-\frac{\lambda^2}{144}\sum_{K\in \cT_h}H_K\parallel u \parallel_{0,K}^2\\
&+\lambda \sum_{K\in\cT_h}\int_K (\bold{x}-\bold{M}_K)P_K^{\rm RT}(\nabla^2 u )u \dx+O( h^{3}|\ln h|^{1/2}|u |_{\frac{7}{2},\Om}^2 ).
\end{split}
\end{equation*}
Furthermore,
$$
\big |\lambda -\lambda_{\rm EXP}^{\rm CR}\big |\lesssim h^{3}|\ln h|^{1/2}|u |_{\frac{7}{2},\Om}^2,
$$
where the extrapolation eigenvalue $\lambda_{\rm EXP}^{\rm CR}=\frac{4\lambda^{\rm (CR, E)}_h-\lambda^{\rm (CR, E)}_{2h}}{3}$.
\end{Th}
\begin{remark}
In \cite{Hu2016Superconvergence}, the superconvergence of the Hellan-Herrmann-Johnson element was analyzed. Since the Morley element is equivalent to the Hellan-Herrmann-Johnson element \cite{Hu2016Superconvergence}, for forth order elliptic eigenvalue problems, asymptotic expansions of eigenvalues by the Morley element can be established and employed to achieve high accuracy extrapolation methods following a similar procedure.
\end{remark}

\section{Asymptotically exact a posteriori error estimators}\label{sec:errorestimate}
In this section, for second order elliptic eigenvalue problems, asymptotically exact a posteriori error estimators of eigenvalues are constructed and analyzed for the CR element and the ECR element.

For eigenvalues of the Laplacian operator solved by the conforming linear element, asymptotically exact a posteriori error estimators were constructed  in \cite{Zhang2006Enhancing}. It is based on a simple identity
$$
\lambda_h -\lambda =\parallel \nabla_h (u -u_h )\parallel_{0,\Om}^2-\lambda  \parallel u -u_h \parallel_{0,\Om}^2.
$$
Since the second term on the right side of the above identity is of higher order, new approximate eigenvalues with high accuracy can be obtained by the gradient recovery techniques \cite{Zhang2005A,Zienkiewicz1992The,huang2010superconvergent}.

For nonconforming elements of second order elliptic eigenvalue problems, the identity becomes
\begin{equation*}
\begin{split}
\lambda -\lambda_h =a_h(u -u_h ,u -u_h )+2 a_h(u ,u_h )-2\lambda_h (u ,u_h )-\lambda_h (u -u_h ,u -u_h ).
\end{split}
\end{equation*}
Compared to conforming elements, the extra term
$$
a_h(u ,u_h )-\lambda_h (u ,u_h )
$$
for nonconforming elements relates to functions themselves. For the CR element and the ECR element, their canonical interpolations of eigenfunctions are employed here to approximate this term with high accuracy by the gradient recovery techniques. To be specific, for the CR element, thanks to the commuting property of the canonical interpolation operator $\Pi_h^{\rm CR}$ in \eqref{commuting},
\begin{equation*}
\begin{split}
\lambda -\lambda_h^{\rm (CR, E)}=&a_h(u -u_h^{\rm (CR, E)},u -u_h^{\rm (CR, E)})-2\lambda_h^{\rm (CR, E)}(u -\Pi^{\rm CR}_h u ,u_h^{\rm (CR, E)})\\
&-\lambda_h^{\rm (CR, E)}(u -u_h^{\rm (CR, E)},u -u_h^{\rm (CR, E)}).
\end{split}
\end{equation*}
The term $a_h(u -u_h^{\rm (CR, E)},u -u_h^{\rm (CR, E)})$ can be approximated with high accuracy by the gradient recovery techniques. Meanwhile, according to Lemma \ref{Lm:cr}, the interpolation error $(I-\Pi^{\rm CR}_h)w $ of any quadratic function $w $ can be expressed in terms of only the second order derivatives of $w $. Therefore, the extra term $\lambda_h^{\rm (CR, E)}(u -\Pi^{\rm (CR, E)}_h u ,u_h^{\rm (CR, E)})$ can also be approximated with high accuracy by the gradient recovery techniques. High accuracy approximate eigenvalues by the ECR element can also be obtained following a similar procedure.

Define the following a posteriori error estimators
\begin{equation}\label{FdefineCR}
F^{\rm CR}_h=\parallel K_h\nabla_h u_h^{\rm (CR, E)} -\nabla_h u^{\rm (CR, E)}_h\parallel_{0,\Om}^2-2\lambda^{\rm (CR, E)}_h \sum_{K\in\cT_h} \int_K P^{\rm CR}_K(\nabla_h K_h\nabla_h u_h^{\rm (CR, E)}) u^{\rm (CR, E)}_h\dx,
\end{equation}
\begin{equation}\label{FdefineECR}
F^{\rm ECR}_h=\parallel K_h\nabla_h u_h^{\rm (ECR, E)} -\nabla_h u^{\rm (ECR, E)}_h\parallel_{0,\Om}^2-2\lambda^{\rm (ECR, E)}_h \sum_{K\in\cT_h} \int_K P^{\rm ECR}_K(\nabla_h K_h\nabla_h u_h^{\rm (ECR, E)}) u^{\rm (ECR, E)}_h\dx.
\end{equation}

\begin{Lm}\label{highorderestimate}
Let $(\lambda , u )$ be the eigenpair of \eqref{variance} with $u \in H^{\frac{7}{2}}(\Om,\mathbb{R})\cap H^1_0(\Om,\mathbb{R})$, and $(\lambda^{\rm (CR, E)}_h,u^{\rm (CR, E)}_h)$ be the corresponding approximate eigenpair of \eqref{discrete} in $V^{\rm CR}_h$. It holds that
\begin{equation*}
\parallel \nabla^{2} u -\nabla_h K_h\nabla_h u_h^{\rm (CR, E)}\parallel_{0,\Om}\lesssim h|\ln h|^{1/2}|u |_{\frac{7}{2},\Om}.
\end{equation*}
\end{Lm}
\begin{proof}
Let $ \Pi_h^2 u $ be the second order Lagrangian interpolation of $u $, namely, the interpolation $ \Pi_h^2 u$ is a piecewise quadratic function over $\cT_h$ and admits the same value as $u $ at the vertices of each element and the midpoint of each edge. It follows from the theory in \cite{ShiWangBook} that
\begin{equation}\label{polationerr}
\big| u -\Pi_h^2 u \big|_{i,\Om}\lesssim h^{3-i}|u |_{3,\Om},\ 0\leq i\leq 2.
\end{equation}
Due to the triangle inequality,
\begin{equation}\label{uRHu2014Lowerhtotal}
\parallel \nabla^2 u -\nabla_h K_h\nabla_h u_h^{\rm (CR, E)}\parallel_{0,\Om}\leq \parallel \nabla^2 u -\nabla^2_h \Pi_h^2 u \parallel_{0,\Om} + \parallel \nabla^2_h \Pi_h^2 u -\nabla_h K_h\nabla_h u_h^{\rm (CR, E)}\parallel_{0,\Om}.
\end{equation}
By the inverse inequality,
\begin{equation}\label{inverseineq}
\parallel \nabla^{2}_h \Pi_h^{2} u -\nabla_h K_h\nabla_h u_h^{\rm (CR, E)}\parallel_{0,\Om}\lesssim h^{-1}\parallel \nabla_h \Pi_h^{2} u -K_h\nabla_h u_h^{\rm (CR, E)} \parallel_{0,\Om}.
\end{equation}
A combination of \eqref{polationerr}, \eqref{inverseineq} and Theorem \ref{supereig1} yields
\begin{equation}\label{Rhinterpolation}
\begin{split}
\parallel\nabla^{2}_h\Pi_h^{2} u -\nabla_h K_h\nabla_h u_h^{\rm (CR, E)}\parallel_{0,\Om}&\lesssim h^{-1}\parallel \nabla_h \Pi_h^{2} u -\nabla u  \parallel_{0,\Om}+h^{-1}\parallel \nabla u -K_h\nabla_h u_h^{\rm (CR, E)} \parallel_{0,\Om}\\
&\lesssim h|\ln h|^{1/2}|u |_{\frac{7}{2},\Om}.
\end{split}
\end{equation}
A substitution of \eqref{polationerr} and \eqref{Rhinterpolation} into \eqref{uRHu2014Lowerhtotal} concludes
\begin{equation*}
\parallel \nabla^{2} u -\nabla_h K_h\nabla_h u_h^{\rm (CR, E)}\parallel_{0,\Om}\lesssim h|\ln h|^{1/2}|u |_{\frac{7}{2},\Om},
\end{equation*}
which completes the proof.
\qed
\end{proof}

The following theorem shows that the a posteriori error estimator $F^{\rm CR}_h$ in \eqref{FdefineCR} is asymptotically exact.

\begin{Th}\label{Th:cr}
Let $(\lambda , u )$ be the eigenpair of \eqref{variance} with $u \in H^{\frac{7}{2}}(\Om,\mathbb{R})\cap H^1_0(\Om,\mathbb{R})$, and $(\lambda^{\rm (CR, E)}_h,u^{\rm (CR, E)}_h)$ be the corresponding approximate eigenpair of \eqref{discrete} in $V^{\rm CR}_h$. The a posteriori error estimator $F_h^{\rm CR}$ in \eqref{FdefineCR} satisfies
$$
\big |\lambda -\lambda^{\rm (CR, E)}_h-F^{\rm CR}_h\big |\lesssim h^{3}|\ln h|^{1/2}|u |_{\frac{7}{2},\Om}^2.
$$
\end{Th}
\begin{proof}
The identity \eqref{commutId} reads
\begin{equation*}
\lambda -\lambda^{\rm (CR, E)}_h=a_h(u -u^{\rm (CR, E)}_h,u -u^{\rm (CR, E)}_h)-2\lambda^{\rm (CR, E)}_h(u -\Pi^{\rm CR}_h u ,u^{\rm (CR, E)}_h)-\lambda^{\rm (CR, E)}_h(u -u^{\rm (CR, E)}_h,u -u^{\rm (CR, E)}_h).
\end{equation*}
By the definition of $F^{\rm CR}_h$,
\begin{equation}\label{errorexpand}
\begin{split}
\lambda -\lambda^{\rm (CR, E)}_h-F^{\rm CR}_h=&a_h(u -u^{\rm (CR, E)}_h,u -u^{\rm (CR, E)}_h) - \parallel K_h\nabla_h u_h^{\rm (CR, E)} -\nabla_h u^{\rm (CR, E)}_h\parallel_{0,\Om}^2\\
& -2\lambda^{\rm (CR, E)}_h\sum_{K\in\cT_h} \big (u -\Pi^{\rm CR}_h u -P^{\rm CR}_K(\nabla^2 u ),u^{\rm (CR, E)}_h\big )_{0,K}\\
&-2\lambda^{\rm (CR, E)}_h\sum_{K\in\cT_h} \big (P^{\rm CR}_K(\nabla^2 u )-P^{\rm CR}_K(\nabla_h K_h\nabla_h u_h^{\rm (CR, E)} ),u^{\rm (CR, E)}_h\big )_{0,K}\\
&-\lambda^{\rm (CR, E)}_h(u -u^{\rm  (CR, E)}_h,u -u^{\rm (CR, E)}_h).
\end{split}
\end{equation}
Thanks to Theorem \ref{supereig1} and \eqref{CR:est},
\begin{equation}\label{1term}
\begin{split}
\big |a_h(u -u^{\rm (CR, E)}_h,u -u^{\rm (CR, E)}_h) - \parallel K_h\nabla_h u_h^{\rm (CR, E)} -\nabla_h u^{\rm (CR, E)}_h\parallel_{0,\Om}^2\big |\lesssim h^{3}|\ln h|^{1/2}|u |_{\frac{7}{2},\Om}^2.
\end{split}
\end{equation}
A combination of the Bramble-Hilbert lemma and Lemma \ref{Lm:cr} leads to
\begin{equation}\label{2term}
\big | \sum_{K\in\cT_h}\big (u -\Pi^{\rm CR}_h u -P^{\rm CR}_K(\nabla^2 u ),u^{\rm (CR, E)}_h\big )_{0,K}\big |\lesssim h^3|u |_{3,\Om}.
\end{equation}
According to Lemma \ref{highorderestimate},
\begin{equation}\label{3term}
\big |\sum_{K\in\cT_h}\big (P^{\rm CR}_K(\nabla^2 u )-P^{\rm CR}_K(\nabla_h K_h\nabla_h u_h^{\rm (CR, E)} ),u^{\rm (CR, E)}_h\big )_{0,K}\big |\lesssim  h^{3}|\ln h|^{1/2}|u |_{\frac{7}{2},\Om}^2 .
\end{equation}
It follows from \eqref{CR:est} that
\begin{equation}\label{4term}
(u -u^{\rm (CR, E)}_h,u -u^{\rm (CR, E)}_h)\lesssim h^4|u |_{2,\Om}^2.
\end{equation}
A substitution of \eqref{1term}, \eqref{2term}, \eqref{3term} and \eqref{4term} into \eqref{errorexpand} concludes
$$
\big |\lambda -\lambda^{\rm (CR, E)}_h-F^{\rm CR}_h\big |\lesssim h^{3}|\ln h|^{1/2}|u |_{\frac{7}{2},\Om}^2,
$$
which completes the proof.
\qed
\end{proof}
Notice that other a posteriori error estimators can be constructed following \eqref{FdefineCR}, but using other recovered gradients from $K_h\nabla_h u_h^{\rm (CR, E)}$. The resulted a posteriori error estimators  are also asymptotically exact as long as the recovered gradients superconverge to the gradients of eigenfunctions.

Similarly, the a posteriori error estimator $F^{\rm ECR}_h$ in \eqref{FdefineECR} is asymptotically exact, as presented in the following theorem.
\begin{Th}\label{Th:ecr}
Let $(\lambda , u )$ be the eigenpair of \eqref{variance} with $u \in H^{\frac{7}{2}}(\Om,\mathbb{R}) \cap  H^1_0(\Om,\mathbb{R})$, and $(\lambda^{\rm (ECR, E)}_h,u^{\rm (ECR, E)}_h)$ be the corresponding approximate eigenpair of \eqref{discrete} in $V^{\rm ECR}_h$. Then,
$$
\big |\lambda -\lambda^{\rm (ECR, E)}_h-F^{\rm ECR}_h\big |\lesssim h^{3}|\ln h|^{1/2}|u |_{\frac{7}{2},\Om}^2.
$$
\end{Th}
\begin{remark}\label{remark:morley}
For fourth order elliptic source problems, let $u_h^{\rm (M,\ f)}$ be the finite element solution of $u^{\rm (f,\ S)}$ by the Morley element, it was analyzed in \cite{Hu2016Superconvergence} that the recovered hessian $K_h\nabla^2_h u_h^{\rm (M,\ f)}$ satisfies $ \parallel K_h\nabla^2_h u_h^{\rm (M,\ f)} -\nabla^2 u^{\rm (f,\ S)}\parallel_{0,\Om}\lesssim h^{3/2}|u^{\rm (f,\ S)}|_{4,\Om}$. Since the canonical interpolation operator of the Morley element also admits a commuting property, a similar procedure produces asymptotically exact a posteriori error estimators for eigenvalues by the Morley element.
\end{remark}
\section{Postprocessing algorithm}\label{sec:algo}
This section proposes two methods to improve accuracy of approximate eigenvalues by employing asymptotically exact a posteriori error estimators.

\begin{Th}\label{th:algo}
\item[\quad $\bullet$]Given an approximate eigenvalue $\lambda_h $ and an a posteriori error estimators $F_h$, which satisfies
\begin{equation*}\label{lam1base}
\lambda =\lambda_h +F_h +O(h^{\gamma}),
\end{equation*}
define a recovering eigenvalue approximation  by
\begin{equation*}\label{lam1}
\lambda_h^{\rm REA}:=\lambda_h +F_h.
\end{equation*}
It holds that
$$
|\lambda -\lambda_h^{\rm REA}|\lesssim h^{\gamma}.
$$
\item[\quad $\bullet$]Given two approximate eigenvalues $\lambda_h^1$ and $\lambda_h^2$, and the corresponding a posteriori error estimators $F_h^1$ and $F_h^2$, which satisfy
\begin{equation*}\label{lam2base}
\lambda =\lambda_h^1+F_h^1 +O(h^{\gamma}),\quad \lambda =\lambda_h^2+F_h^2 +O(h^{\gamma}),
\end{equation*}
define a combining eigenvalue approximation by
\begin{equation}\label{lam2}
\lambda_h^{\rm CEA}:=\frac{F_h^2}{F_h^2-F_h^1}\lambda_h^1-\frac{F_h^1}{F_h^2-F_h^1}\lambda_h^2.
\end{equation}
It holds that
$$
|\lambda -\lambda_h^{\rm CEA}|\lesssim h^{\gamma}.
$$
\end{Th}

The combining eigenvalue approximation $\lambda_h^{\rm CEA}$ in \eqref{lam2} is a weighted-average  of two approximate eigenvalues, and of high accuracy. Different from the construction in \cite{Hu2012A}, the weights here are computed by the corresponding a posteriori error estimators $F_h^1$ and $F_h^2$, instead of by solving the eigenvalue problem by two elements, which produce two upper bounds and two lower bounds of eigenvalues, respectively, on two successive meshes.

Next, we propose a new way to construct combining eigenvalue approximations with high accuracy by solving only one discrete eigenvalue problem. To this end, first solve the eigenvalue problem by the CR element which produces lower bounds of eigenvalues, and denote the resulted eigenpair by $(\lambda_h^{\rm (CR, E)}, u_h^{\rm (CR, E)})$. Then an application of the average-projection in \cite{Hu2015Constructing} to the approximate eigenfunction  $u_h^{\rm (CR, E)}$ results in a conforming function $\tilde{u}_h^{\rm (P_1^{\ast}, E)}$. Next, define
\begin{equation}\label{projectionP1}
u_h^{\rm (P_1^{\ast}, E)}:= \tilde{u}_h^{\rm (P_1^{\ast}, E)}/\parallel \tilde{u}_h^{\rm (P_1^{\ast}, E)}\parallel_{0,\Om}\ \text{ and }\ \lambda_h^{\rm  (P_1^{\ast}, E)}:=a_h(u_h^{\rm (P_1^{\ast}, E)},u_h^{\rm (P_1^{\ast}, E)}).
\end{equation}
According to \cite{Hu2015Constructing}, $u_h^{\rm (P_1^{\ast}, E)}$ is a conforming approximation of the eigenfunction $u$, and the Rayleigh quotient $\lambda_h^{\rm (P_1^{\ast}, E)}$ is an asymptotical upper bound of the eigenvalue $\lambda$. For the two approximate eigenpairs $(\lambda_h^{\rm (CR, E)}, u_h^{\rm (CR, E)})$ and $(\lambda_h^{\rm (P_1^{\ast}, E)}, u_h^{\rm (P_1^{\ast}, E)})$, following the procedure in Section \ref{sec:errorestimate}, we can construct the corresponding asymptotically exact a posteriori error estimators $F_h^{\rm (CR, E)}$ and $F_h^{\rm (P_1^{\ast}, E)}$, respectively. Finally, define a new approximation
\begin{equation}\label{combineEig}
\lambda_h^{\rm CEA}:=\frac{F_h^{\rm (CR, E)}}{F_h^{\rm (CR, E)}-F_h^{\rm (P_1^{\ast}, E)}}\lambda_h^{\rm (P_1^{\ast}, E)}-\frac{F_h^{\rm (P_1^{\ast}, E)}}{F_h^{\rm (CR, E)}-F_h^{\rm (P_1^{\ast}, E)}}\lambda_h^{\rm (CR, E)}.
\end{equation}
Note that the high accuracy of the resulted approximate eigenvalue $\lambda_h^{\rm CEA}$ in \eqref{combineEig} is guaranteed by Theorem \ref{th:algo}.

\section{Numerical examples}\label{sec:numerical}
This section presents five numerical tests. The first four examples compute eigenvalues of the Laplacian operator, and the last one deals with eigenvalues of the biharmonic operator.

\subsection{Example 1.}
In this example, the model problem \eqref{variance} on the unit square $\Om=(0,1)^2$ is considered. In this case, the exact eigenvalues are
$$\lambda =(m^2+n^2)\pi^2 ,\ m,\ n\ \text{are positive integers},$$
and the corresponding eigenfunctions are $u =2\sin (m\pi x_1)\sin (n\pi x_2)$. The domain is partitioned by uniform triangles. The level one triangulation $\cT_1$ consists of two right triangles, obtained by cutting the unit square with a north-east line. Each triangulation $\cT_i$ is refined into a half-sized triangulation uniformly, to get a higher level triangulation $\cT_{i+1}$.

Denote the approximate eigenpairs by the CR element, the ECR element, the conforming linear element on $\cT_h$ by $(\lambda_h^{\rm (CR, E)}, u_h^{\rm (CR, E)})$, $(\lambda_h^{\rm (ECR, E)}, u_h^{\rm (ECR, E)})$ and  $(\lambda_h^{\rm (P_1, E)}, u_h^{\rm (P_1, E)})$, respectively. The approximate eigenpair $(\lambda_h^{\rm (P_1^{\ast}, E)}, u_h^{\rm (P_1^{\ast}, E)})$ on $\cT_h$ is defined in \eqref{projectionP1}.



\subsubsection{Recovering eigenvalues}
Denote the recovering eigenvalue $\lambda_{\rm CR}^{\rm R,\ P_1^{\ast}}=\lambda_h^{\rm (CR, E)}+F_{\rm CR}^{\rm P_1^{\ast}}$ with the following asymptotically exact a posteriori error estimator
\begin{equation*}\label{FdefineCR}
F_{\rm CR}^{\rm P_1^{\ast}}:=\parallel \overline{K}_h\nabla_h u_h^{\rm (P_1^{\ast}, E)} -\nabla_h u^{\rm (CR, E)}_h\parallel_{0,\Om}^2-2\lambda^{\rm (CR, E)}_h \sum_{K\in\cT_h} \int_K P^{\rm CR}_K(\nabla_h \overline{K}_h\nabla_h u_h^{\rm (P_1^{\ast}, E)}) u^{\rm (CR, E)}_h\dx,
\end{equation*}
where the operator $\overline{K}_h$ refers to the PPR technique in \cite{Zhang2005A}. Let the recovering eigenvalue $\lambda_{\rm P_1^{\ast}}^{\rm R,\ CR}=\lambda_h^{\rm (P_1^{\ast}, E)}+F_{\rm P_1^{\ast}}^{\rm CR}$ with the following asymptotically exact a posteriori error estimator
$$
F_{\rm P_1^{\ast}}^{\rm CR}:=\parallel K_h\nabla_h u_h^{\rm (CR, E)} -\nabla_h u^{\rm (P_1^{\ast}, E)}_h\parallel_{0,\Om}^2.
$$
The other recovering eigenvalues and a posteriori error estimators are defined in a similar way.

Figure \ref{fig:squareCRECRP1} plots the errors of the first approximate eigenvalues by the CR element, the ECR element, the conforming linear element and their corresponding recovering eigenvalues.
\begin{figure}[!ht]
\setlength{\abovecaptionskip}{0pt}
\setlength{\belowcaptionskip}{0pt}
\centering
\includegraphics[width=11cm,height=9cm]{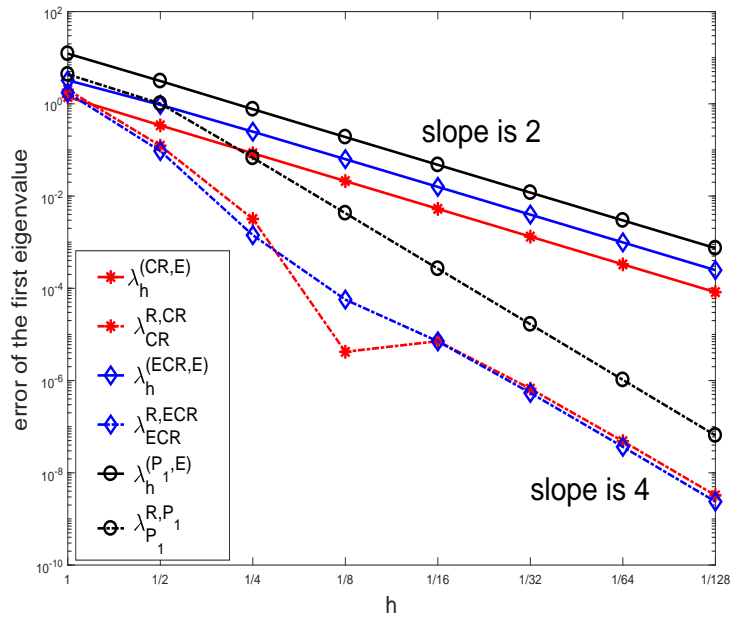}
\caption{\footnotesize{The errors of recovering eigenvalues for Example 1.}}
\label{fig:squareCRECRP1}
\end{figure}
It shows that the approximate eigenvalues $\lambda_h^{\rm (CR, E)}$, $\lambda_h^{\rm (ECR, E)}$ and $\lambda_h^{\rm (P_1, E)}$ converge at a rate 2, and the recovering eigenvalues $\lambda_{\rm CR}^{\rm R,\ CR}$, $\lambda_{\rm ECR}^{\rm R,\ ECR}$ and $\lambda_{\rm P_1}^{\rm R,\ P_1}$ converge at a higher rate 4. Note that although the theoretical convergence rates of the recovering eigenvalues are only 3, numerical tests indicate that the convergence rates are 4. The errors of the recovering eigenvalues $\lambda_{\rm CR}^{\rm R,\ CR}$, $\lambda_{\rm ECR}^{\rm R,\ ECR}$ and $\lambda_{\rm P_1}^{\rm R,\ P_1}$ on $\cT_8$  are $3.25\times 10^{-9}$, $2.38\times 10^{-9}$ and $6.47\times 10^{-8}$, respectively, they are significant improvements on the errors of the approximate eigenvalues $\lambda_h^{\rm (CR, E)}$, $\lambda_h^{\rm (ECR, E)}$ and $\lambda_h^{\rm (P_1, E)}$, which are $8.26\times 10^{-5}$, $ 2.48\times 10^{-4}$ and $7.43\times 10^{-4}$, respectively. This reveals that recovering eigenvalues are quite remarkable improvements on finite element solutions.

Table \ref{tab:variedpost} compares the errors of different recovering eigenvalues. It shows that on each mesh, the most accurate approximation is $\lambda_{\rm CR}^{\rm R,\ CR}$. Meanwhile, the errors of $\lambda_{\rm P_1}^{\rm R,\ CR}$ and $\lambda_{\rm P_1^{\ast}}^{\rm R,\ CR}$ are almost the same, and they are smaller than the other errors except that of $\lambda_{\rm CR}^{\rm R,\ CR}$. Note that only one discrete eigenvalue problem needs to be computed for the recovering eigenvalue $\lambda_{\rm P_1^{\ast}}^{\rm R,\ CR}$, but two for the recovering eigenvalue $\lambda_{\rm P_1}^{\rm R,\ CR}$.
\begin{table}[!ht]
  \centering
    \begin{tabular}{c|cccccc}
    \hline
        h &$\lambda_{\rm CR}^{\rm R,\ CR}$&$\lambda^{\rm R,\ P_1^{\ast}}_{\rm CR}$&$ \lambda_{\rm P_1}^{\rm R,\ CR}$&$\lambda_{\rm P_1}^{\rm R,\ P_1}$&$\lambda_{\rm P_1^{\ast}}^{\rm R,\ CR}$&$ \lambda_{\rm P_1^{\ast}}^{\rm R,\ P_1^{\ast}}$\\\hline
       $1/2$& 2.0124  & -1.1152 & 76.2332  & 4.3620  & 9.5561  & 4.3620 \\\hline
       $1/4$& 0.1236  & -0.7206& 0.4255  & 1.0103  & 0.4192  & 1.1415  \\\hline
       $1/8$& 3.19E-03 & -7.12E-02 & 2.33E-02 & 6.77E-02 & 2.06E-02 & 9.20E-02\\\hline
       $1/16$& -4.20E-06 & -5.63E-03 & 1.31E-03 & 4.26E-03 & 1.10E-03 & 7.32E-03 \\\hline
       $1/32$& -7.15E-06 & -4.55E-04 & 7.67E-05 & 2.66E-04 & 6.29E-05 & 6.39E-04\\\hline
       $1/64$& -6.65E-07 & -4.06E-05 & 4.63E-06 & 1.66E-05 & 3.75E-06 & 6.24E-05 \\\hline
       $1/128$& -4.84E-08 & -4.04E-06& 2.84E-07 & 1.04E-06 & 2.29E-07 & 6.69E-06 \\\hline
       $1/256$& -3.25E-09 & -4.41E-07& 1.76E-08 & 6.47E-08 & 1.41E-08 & 7.68E-07  \\\hline
    \end{tabular}%
  \caption{\footnotesize  The errors of different recovering eigenvalues for Example 1.}
  \label{tab:variedpost}%
\end{table}%

\subsubsection{Combining eigenvalues}
A combining eigenvalue approximation involves two different approximate eigenvalues, and also two asymptotically exact a posteriori error estimators. In this part, the weighted-average of a lower bound and an upper bound of the eigenvalue is considered. The lower bound is chosen to be $\lambda_h^{\rm (CR, E)}$, and the upper bound is $\lambda_h^{\rm (P_1, E)}$ or $\lambda_h^{\rm (P_1^{\ast}, E)}$. The combining eigenvalue $\lambda^{\rm C,\ P_1^{\ast}}_{\rm P_1,CR}$ is a weighted-average of the eigenvalues $\lambda_h^{\rm (CR, E)}$ and $\lambda_h^{\rm (P_1^{\ast}, E)}$, and the asymptotically exact a posteriori error estimator for the former approximate eigenvalue is $F_{\rm CR}^{\rm P_1}$, the one for the latter approximate eigenvalue is $F_{\rm P_1^{\ast}}^{\rm CR}$, namely,
$$
\lambda^{\rm C,\ P_1^{\ast}}_{\rm P_1,CR}=\frac{F_{\rm CR}^{\rm P_1}}{F_{\rm CR}^{\rm P_1}-F_{\rm P_1^{\ast}}^{\rm CR}}\lambda_h^{\rm (P_1^{\ast}, E)}-\frac{F_{\rm P_1^{\ast}}^{\rm CR}}{F_{\rm CR}^{\rm P_1}-F_{\rm P_1^{\ast}}^{\rm CR}}\lambda_h^{\rm (CR, E)}.
$$
The other combining eigenvalues are defined in a similar way.

The errors of some combining eigenvalues on $\cT_8$ are recorded in Table \ref{tab:combination}. Among all the errors in Table \ref{tab:combination}, the smallest one is $1.17\times 10^{-9}$, and it is the error of a weighted-average of $\lambda_h^{\rm (CR, E)}$ and $\lambda_h^{\rm (P_1, E)}$, where the weights are computed by $F_{\text{CR}}^{\text{CR}}$ and $F_{\text{P}_1}^{\text{CR}}$. The combining eigenvalue proposed in Section \ref{sec:algo} is a weighted-average of $\lambda_h^{\rm (CR, E)}$ and $\lambda_h^{\rm (P_1^{\ast}, E)}$, the weights are computed by $F_{\text{CR}}^{\text{CR}}$ and $F_{\text{P}_1^{\ast}}^{\text{CR}}$. The error of this combining eigenvalue on $\cT_8$ is $1.51\times 10^{-9}$, only slightly larger than the smallest error in Table \ref{tab:combination}.
\begin{table}[!ht]
  \centering
    \begin{tabular}{c|cccc}
    \hline
    &$\lambda^{\rm C,\ P_1}_{\rm CR, CR}$ & $\lambda^{\rm C,\ P_1}_{\rm CR, P_1}$ & $\lambda^{\rm C,\ P_1^{\ast}}_{\rm CR, CR}$ &$\lambda^{\rm C,\ P_1^{\ast}}_{\rm CR, P_1^{\ast}}$ \\\hline
    error &-1.17E-09 & 3.55E-09 & -1.51E-09 & 7.38E-08\\\hline
    &$\lambda^{\rm C,\ P_1}_{\rm P_1, CR}$ & $\lambda^{\rm C,\ P_1}_{\rm P_1, P_1}$ & $\lambda^{\rm C,\ P_1^{\ast}}_{\rm P_1, CR}$ &$\lambda^{\rm C,\ P_1^{\ast}}_{\rm P_1, P_1^{\ast}}$\\\hline
    error &-8.36E-08 & -7.89E-08 & -8.39E-08 & -8.60E-09\\\hline
    \end{tabular}
    \caption{\footnotesize The errors of different combining eigenvalues on the mesh $\cT_8$ for Example 1.}
    \label{tab:combination}%
\end{table}

\subsubsection{Extrapolation eigenvalues}
\begin{figure}[!ht]
\setlength{\abovecaptionskip}{0pt}
\setlength{\belowcaptionskip}{0pt}
\centering
\includegraphics[width=11cm,height=9cm]{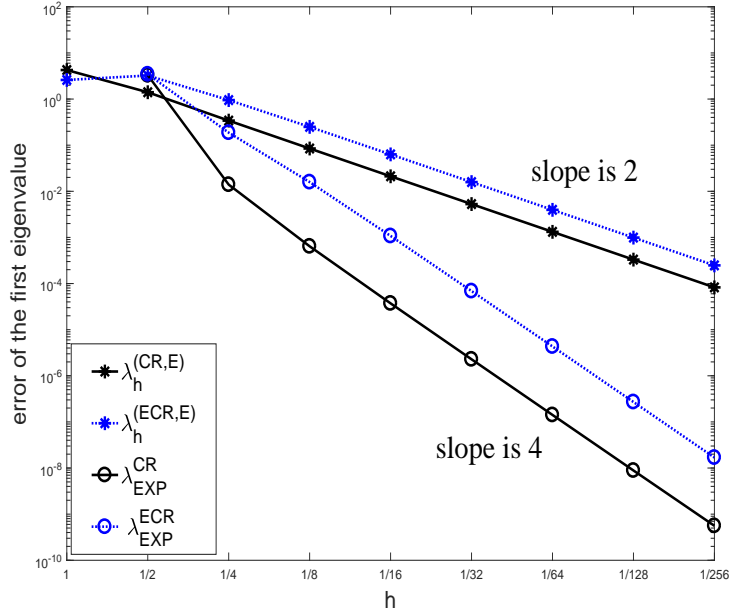}
\caption{\footnotesize{The errors of extrapolation eigenvalues on an uniform triangulation for Example 1.}}
\label{fig:squareExtrapolation}
\end{figure}

Figure \ref{fig:squareExtrapolation} plots the errors of the first approximate eigenvalues by the CR element, the ECR element and their corresponding extrapolation eigenvalues on the aforementioned uniform triangulations. As showed in Figure \ref{fig:squareExtrapolation}, the convergence rate 3 of the extrapolation eigenvalues $\lambda_{\rm EXP}^{\rm CR}$ and $\lambda_{\rm EXP}^{\rm ECR}$ in Theorem \ref{Th:extraCR} and Theorem \ref{Th:extraECR} is verified. However, the numerical results indicate a higher convergence rate 4. Table \ref{tab:compareextrapolation} compares the performance of recovering eigenvalues and extrapolation eigenvalues. It shows that the recovering eigenvalue $\lambda_{\rm CR}^{\rm R,\ CR} $ behaves better than the extrapolation eigenvalue $\lambda_{\rm EXP}^{\rm P_1}$, but worse than $ \lambda_{\rm EXP}^{\rm CR}$.

\renewcommand\arraystretch{1.8}
\begin{table}[!ht]
  \centering
    \begin{tabular}{c|cccccc}
    \hline
    h & $\lambda_h^{\rm (CR, E)}$ &$ \lambda_h^{\rm (P_1, E)}$ & $\lambda^{\rm CR}_{\rm EXP}$ & $\lambda^{\rm P_1}_{\rm EXP}$ & $\lambda_{\rm CR}^{\rm R,\ CR}$ & $\lambda_{\rm P_1}^{\rm R,\ P_1}$ \\\hline
       1/4   & -0.3407  & 3.1266  & 0.0140  & 0.0818  & 0.1236  & 1.0103  \\\hline
       1/8   & -8.47E-02 & 7.66E-01 & 6.42E-04 & -2.04E-02 & 3.19E-03 & 6.77E-02 \\\hline
       1/16  & -2.11E-02 & 1.91E-01 & 3.72E-05 & -1.34E-03 & -4.20E-06 & 4.26E-03 \\\hline
       1/32  & -5.29E-03 & 4.76E-02 & 2.28E-06 & -8.24E-05 & -7.15E-06 & 2.66E-04 \\\hline
       1/64  & -1.32E-03 & 1.19E-02 & 1.42E-07 & -5.12E-06 & -6.65E-07 & 1.66E-05 \\\hline
       1/128 & -3.30E-04 & 2.97E-03 & 8.85E-09 & -3.19E-07 & -4.84E-08 & 1.04E-06 \\\hline
       1/256 & -8.26E-05 & 7.43E-04 & 5.59E-10 & -1.99E-08 & -3.25E-09 & 6.47E-08 \\\hline
    \end{tabular}
  \caption{\footnotesize{The errors of recovering eigenvalues and extrapolation eigenvalues, where $\lambda^{\rm P_1}_{\rm EXP}=(4\lambda^{\rm (P_1, E)}_h - \lambda^{\rm (P_1, E)}_{2h})/3$.}}%
  \label{tab:compareextrapolation}%
\end{table}%

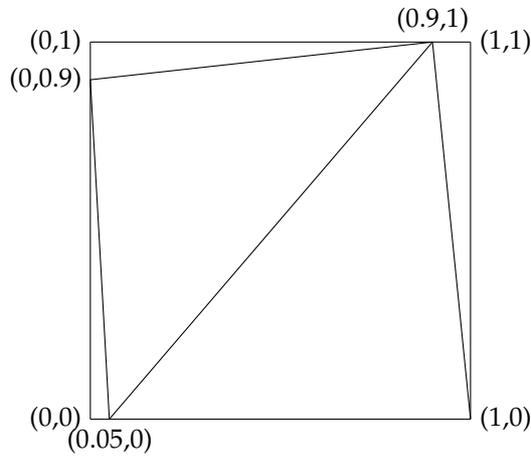
\begin{figure}[!ht]
\begin{center}
\begin{tikzpicture}[xscale=5,yscale=5]
\draw[-] (0,0) -- (0,1);
\draw[-] (0,0) -- (1,0);
\draw[-] (1,0) -- (1,1);
\draw[-] (0,1) -- (1,1);
\draw[-] (0,0.9) -- (0.9,1);
\draw[-] (0,0.9) -- (0.05,0);
\draw[-] (0.05,0) -- (0.9,1);
\draw[-] (1,0) -- (0.9,1);
\node[below, left] at (0,0) {(0,0)};
\node[below, right] at (1,0) {(1,0)};
\node[above, left] at (0,1) {(0,1)};
\node[above, right] at (1,1) {(1,1)};
\node[left] at (0,0.9) {(0,0.9)};
\node[below] at (0.05,0) {(0.05,0)};
\node[above] at (0.9,1) {(0.9,1)};
\end{tikzpicture}
\caption{\footnotesize A level one triangulation $\cT_1$ of $\Om$.}
\label{fig:nonuniformmesh}
\end{center}
\end{figure}
The eigenvalue problem is also solved on other triangulations. The level one triangulation $\cT_1$ is showed in Figure \ref{fig:nonuniformmesh}. Each triangulation $\cT_i$ is refined into a half-sized triangulation uniformly to get a higher level triangulation $\cT_{i+1}$. The errors of the approximate eigenvalues $\lambda_h^{\rm (CR, E)}$, $\lambda_h^{\rm (ECR, E)}$, $\lambda_{\rm EXP}^{\rm CR}$ and $\lambda_{\rm EXP}^{\rm ECR}$ are recorded in Table \ref{tab:NonUniExtra}. It shows that on such triangulations, which are not uniform any more, the convergence rates of the extrapolation eigenvalues are still over 3.
\renewcommand\arraystretch{1.5}
\begin{table}[!ht]
\small
  \centering
    \begin{tabular}{c|ccccccc}
    \hline
          & $\cT_2$& $\cT_3$& $\cT_4$ & $\cT_5$& $\cT_6$&$\cT_7$&  $\cT_8$\\\hline
    $|\lambda -\lambda_h^{\rm (CR, E)}|$    & 0.928068 & 2.22E-01 & 5.55E-02 & 1.39E-02 & 3.48E-03 & 8.69E-04 & 2.17E-04 \\
    $|\lambda -\lambda_{\rm EXP}^{\rm CR}|$& 2.870925 & 1.39E-02 & 7.97E-05 & 3.45E-05 & 3.55E-06 & 3.04E-07 & 2.39E-08 \\
    rate  &       & 7.69  & 7.45  & 1.21  & 3.28  & 3.55  & 3.67  \\\hline
    $|\lambda -\lambda_h^{\rm (ECR, E)}|$   & 2.683924 & 7.68E-01 & 2.01E-01 & 5.07E-02 & 1.27E-02 & 3.18E-03 & 7.96E-04 \\
    $|\lambda -\lambda_{\rm EXP}^{\rm ECR}|$      & 2.825196 & 1.30E-01 & 1.12E-02 & 7.78E-04 & 5.08E-05 & 3.27E-06 & 2.09E-07 \\
    rate  &       & 4.44  & 3.53  & 3.85  & 3.94  & 3.96  & 3.96  \\\hline
    \end{tabular}%
  \caption{\footnotesize The errors and convergence rates of extrapolation eigenvalues on nonuniform triangulations for Example 1.}
  \label{tab:NonUniExtra}%
\end{table}%
\subsection{Example 2}
Next we consider the following eigenvalue problem:
\begin{equation}
\begin{split}
-\Delta u \ &=\ \lambda  u \  \text{\quad in}\ \Om=(0,1)^2 ,\\
u |_{x_1=0} &=u |_{x_2=0}=u |_{x_2=1}=\partial_{x_1} u |_{x_1=1}=0 ,
\end{split}
\end{equation}
In this case, there exists an eigenpair $(\lambda , u )$ where
$$
\lambda =\frac{5\pi^2}{4},\quad u =2\cos \frac{\pi (x_1-1)}{2}\sin \pi x_2.
$$
We solve this problem on the same sequence of uniform triangulations employed in Example 1. Figure \ref{fig:squareNeumannCRECRP1} shows that the approximate eigenvalues by the CR element, the ECR element and the conforming linear element converge at the same rate 2, the recovering eigenvalues  $\lambda_{\rm CR}^{\rm R,\ CR}$ and $\lambda_{\rm P_1}^{\rm R,\ P_1}$ converge  at rate 4. Especially, the recovering eigenvalue $\lambda_{\rm ECR}^{\rm R,\ ECR}$ converges at a strikingly higher rate 6.
\begin{figure}[!ht]
\setlength{\abovecaptionskip}{0pt}
\setlength{\belowcaptionskip}{0pt}
\centering
\includegraphics[width=11cm,height=9cm]{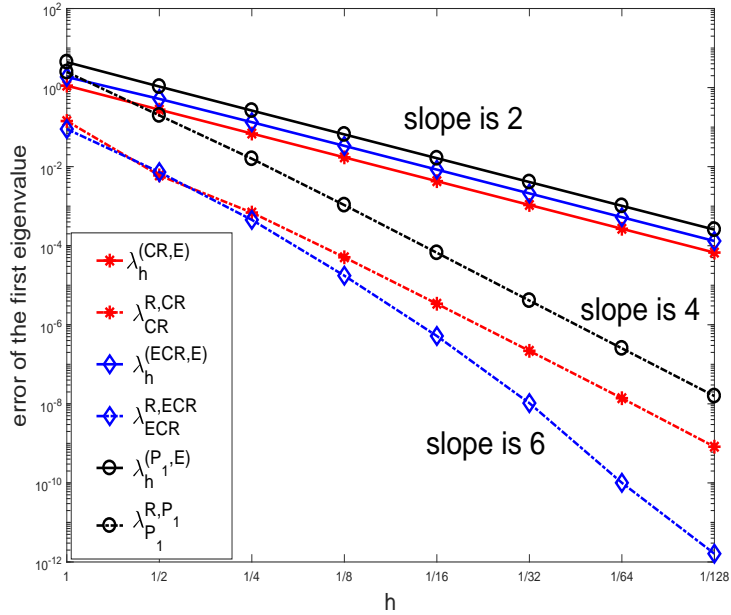}
\caption{\footnotesize{The errors of recovering eigenvalues for Example 2.}}
\label{fig:squareNeumannCRECRP1}
\end{figure}
\subsection{Example 3}
In this experiment, we consider the eigenvalue problem \eqref{variance} on the domain which is an equilateral triangle:
$$
\Om=\big \{(x_1,x_2)\in \mathbb{R}^2:0< x_2< \sqrt{3}x_1, \sqrt{3}(1-x_1)<x_2\big \}.
$$
The boundary consists of three parts:
$
\Gamma_1=\big\{(x_1,x_2)\in \mathbb{R}^2: x_2=\sqrt{3}x_1,\ 0.5\le x_1\le 1 \big\},
$
$
\Gamma_2=\big\{(x_1,x_2)\in \mathbb{R}^2: x_2=\sqrt{3}(1-x_1),\ 0.5\le x_1\le 1 \big\},
$
$
\Gamma_3=\big\{(x_1,x_2)\in \mathbb{R}^2: x_1=1 ,\ 0\le x_2\le 1 \big\}.
$
Under the boundary condition
\begin{equation*}
\left\{
\begin{aligned}
u |_{\Gamma_1\cup \Gamma_2}&=0 \\
\partial_{x_1} u |_{\Gamma_3}&=0 \\
\end{aligned}\ ,
\right.
\end{equation*}
there exists an eigenpair $(\lambda  , u )$, where $\lambda =\frac{16\pi^2}{3}$ and
$$
u =\frac{2\sqrt[4]{12}}{3}\big ( \sin \frac{4\pi x_2}{\sqrt{3}} + \sin 2\pi(x_1-\frac{x_2}{\sqrt{3}}) + \sin 2\pi(1-x_1-\frac{x_2}{\sqrt{3}})\big ).
$$
\renewcommand\arraystretch{1.5}
\begin{table}[!ht]
\small
  \centering
    \begin{tabular}{c|ccccccc}
    \hline
          & $\cT_1$ &$\cT_2$ &$\cT_3$ & $\cT_4$ & $\cT_5$ &$\cT_6$ &$\cT_7$ \\\hline
    $\lambda -\lambda_h^{\rm (CR, E)}$& -3.3545 & 3.6697 & 9.61E-01 & 2.43E-01 & 6.10E-02 & 1.53E-02 & 3.82E-03 \\
    $\lambda -\lambda_{\rm CR}^{\rm R,\ CR}$& -23.0992 & -8.74E-03 & 5.69E-03 & 7.02E-04 & 5.58E-05 & 3.86E-06 & 2.53E-07 \\
    rate   &  -    & 11.37 & 0.62  & 3.02  & 3.65  & 3.85  & 3.93 \\\hline
    $\lambda -\lambda_h^{\rm (ECR, E)}$& 1.177 & 4.8974 & 1.3186 & 3.36E-01 & 8.44E-02 & 2.11E-02 & 5.28E-03 \\
    $\lambda -\lambda_{\rm ECR}^{\rm R,\ ECR}$& -19.2136 & 0.3463 & 2.85E-02 & 2.27E-03 & 1.55E-04 & 1.01E-05 & 6.40E-07 \\
    rate  &  -    & 5.79  & 3.6   & 3.65  & 3.87  & 3.94  & 3.97 \\\hline\
    $\lambda -\lambda_h^{\rm (P_1, E)}$& 20.1024 & -11.1936 & -2.879 & -7.29E-01 & -1.83E-01 & -4.58E-02& -1.14E-02\\
    $\lambda -\lambda_{\rm P_1}^{\rm R,\ P_1}$& 26.817 & 1.0209 & 0.1005 & 3.49E-03 & 9.62E-05 & 2.21E-06 & 1.93E-08 \\
    rate  & -     &  4.72 & 3.35  & 4.85  & 5.18  & 5.45  & 6.84 \\\hline
    \end{tabular}%
  \caption{\footnotesize The errors and convergence rates of recovering eigenvalues for Example 3.}
    \label{tab:triangleCRECRP1}%
\end{table}%
The level one triangulation $\cT_1$ is obtained by refining the domain $\Om$ into four half-sized triangles. Each triangulation $\cT_i$ is refined into a half-sized triangulation uniformly, to get a higher level triangulation $\cT_{i+1}$. It is showed in Table \ref{tab:triangleCRECRP1} that the convergence rates of the recovering eigenvalues $\lambda_{\rm CR}^{\rm R,\ CR}$ and $\lambda_{\rm ECR}^{\rm R,\ ECR}$ are 4.


\subsection{Example 4}
Next we consider the following eigenvalue problem
\begin{equation}
\begin{split}
-\Delta u \ &=\ \lambda  u \  \text{\quad in}\ \Om ,\\
u  &=0 \  \text{ \quad in }\ \partial \Om,
\end{split}
\end{equation}
on a L-shaped domain $\Om=(-1,1)^2/[0,1]\times[-1,0]$. For this problem, the third and the eighth eigenvalues are known to be $2\pi^2$ and $4\pi^2$, respectively, and the corresponding eigenfunctions are smooth.

In the computation, the level one triangulation is obtained by dividing the domain into three unit squares, each of which is further divided into two triangles. Each triangulation is refined into a half-sized triangulation uniformly to get a higher level triangulation. Since exact eigenvalues of this problem are unknown, we solve the first eight eigenvalues by the conforming $\rm P_3$ element on the mesh $\cT_9$, and take them as reference eigenvalues.

\begin{center}
\setlength{\tabcolsep}{1.2mm}{
\begin{table}[!ht]
\small
\renewcommand\arraystretch{1.8}
    \begin{tabular}{c|cccccccc}
    \hline
          &$\lambda_1 $ & $\lambda_2 $ & $\lambda_3 $ & $\lambda_4 $ & $\lambda_5 $ & $\lambda_6 $ & $\lambda_7 $ & $\lambda_8 $  \\\hline
    $\lambda_h^{(\rm CR, E)}$ & 3.80E-04 & 2.43E-05 & 1.67E-05 & 4.49E-05 & 3.20E-04 & 2.53E-04 & 1.11E-04 & 8.70E-05 \\\hline
    $\lambda_h^{(\rm P_1, E)}$ & -3.98E-04 & -1.13E-04 & -1.51E-04 & -2.21E-04 & -4.33E-04 & -3.96E-04 & -3.02E-04 & -3.07E-04 \\\hline
    $\lambda_h^{(\rm P_1^{\ast}, E)}$ & -5.20E-04 & -1.13E-04 & -1.51E-04 & -2.21E-04 & -5.24E-04 & -4.48E-04 & -3.03E-04 & -3.39E-04 \\\hline
    $\lambda_{\rm CR}^{\rm R,\ CR}$ & 1.73E-04 & 1.43E-07 & -2.57E-09 & 1.95E-08 & 1.28E-04 & 7.41E-05 & 2.47E-07 & -5.95E-09 \\\hline
    $\lambda_{\rm P_1}^{\rm R,\ P_1}$ & 1.95E-04 & 1.08E-06 & 5.00E-08 & 3.38E-07 & 1.44E-04 & 8.36E-05 & 2.17E-06 & 2.57E-07 \\\hline
    $\lambda_{\rm P_1^{\ast}}^{\rm R,\ CR}$ & 1.75E-04 & 1.02E-07 & 1.15E-08 & 5.16E-08 & 1.30E-04 & 7.51E-05 & 2.51E-07 & 9.30E-08\\\hline
    $\lambda^{\rm C,\ P_1}_{\rm CR, P_1}$ & 1.79E-04 & 1.02E-07 & 2.69E-09 & 7.32E-08 & 1.32E-04 & 7.67E-05 & 7.61E-07 & 5.21E-08\\\hline
    $\lambda^{\rm C,\ P_1^{\ast}}_{\rm CR, P_1^{\ast}}$ & 1.74E-04 & 3.08E-07 & -1.16E-09 & 2.49E-08 & 1.28E-04 & 7.43E-05 & 2.48E-07 & 1.35E-08 \\\hline
    $\lambda^{\rm HHS}_h$ & -2.51E-05 & -1.83E-07 &  -1.89E-10 &  -3.53E-08  & -3.99E-05 & -2.69E-05 & -2.54E-07 & 6.16E-09 \\\hline
    \end{tabular}%
  \caption{\footnotesize Relative errors of different approximations to the first eight eigenvalues on $\cT_7$ for Example 4.}
  \label{tab:Lshape}%
\end{table}}
\end{center}
An application of the post-processing technique in \cite{Hu2012A} to the discrete eigenvalues  by the CR element and the conforming linear element on $\cT_6$ and $\cT_7$ results in a new approximate eigenvalue, denoted by $\lambda^{\rm HHS}_h$. Table \ref{tab:Lshape} compares the relative errors of the first eight approximate eigenvalues on $\cT_7$ by different methods. It implies that the errors of the recovering eigenvalues $\lambda_{\rm CR}^{\rm R,\ CR}$ are slightly smaller than those of $\lambda_{\rm P_1}^{\rm R,\ P_1}$. Meanwhile, the errors of the combining eigenvalues $\lambda^{\rm C,\ P_1}_{\rm CR, P_1}$ and $\lambda^{\rm C,\ P_1^{\ast}}_{\rm CR, P_1^{\ast}}$ are similar to those of the recovering eigenvalues $\lambda_{\rm CR}^{\rm R,\ CR}$, and are slightly larger than those of $\lambda^{\rm HHS}_h$.

It is observed in Table \ref{tab:Lshape} that for different eigenvalues, the relative errors of the approximate eigenvalues $\lambda_h^{\rm (CR, E)}$  do not vary much. This phenomenon still holds for the approximate eigenvalues $\lambda_h^{\rm (P_1, E)}$ and $\lambda_h^{\rm (P_1^{\ast}, E)}$. However, for the approximate eigenvalues $\lambda_{\rm P_1}^{\rm R,\ P_1}$, $\lambda_{\rm P_1^{\ast}}^{\rm R,\ CR}$, $\lambda^{\rm C,\ P_1}_{\rm CR, P_1}$, $\lambda^{\rm C,\ P_1^{\ast}}_{\rm CR, P_1^{\ast}}$ and $\lambda_h^{\rm HHS}$, the relative errors of various eigenvalues are quite different. The reason is that the accuracy of a posteriori error estimators relies on the regularity of corresponding eigenfunctions. Thus, these approximate eigenvalues achieve better accuracy if corresponding eigenfunctions are smooth. Note that the approximate eigenvalues $\lambda_h^{\rm HHS}$ permit higher accuracy than recovering eigenvalues.

\subsection{Example 5}
In this experiment, we consider the following fourth order elliptic eigenvalue problem
\begin{equation}
\begin{split}
\Delta^2 u \ &=\ \lambda  u \  \text{\quad in}\ \Om=(0,1)^2 ,\\
u |_{\partial \Om} &=0 \  \text{,\quad}\ \Delta u |_{\partial \Om}=0.
\end{split}
\end{equation}

\begin{figure}[!ht]
\setlength{\abovecaptionskip}{0pt}
\setlength{\belowcaptionskip}{0pt}
\centering
\includegraphics[width=11cm,height=9cm]{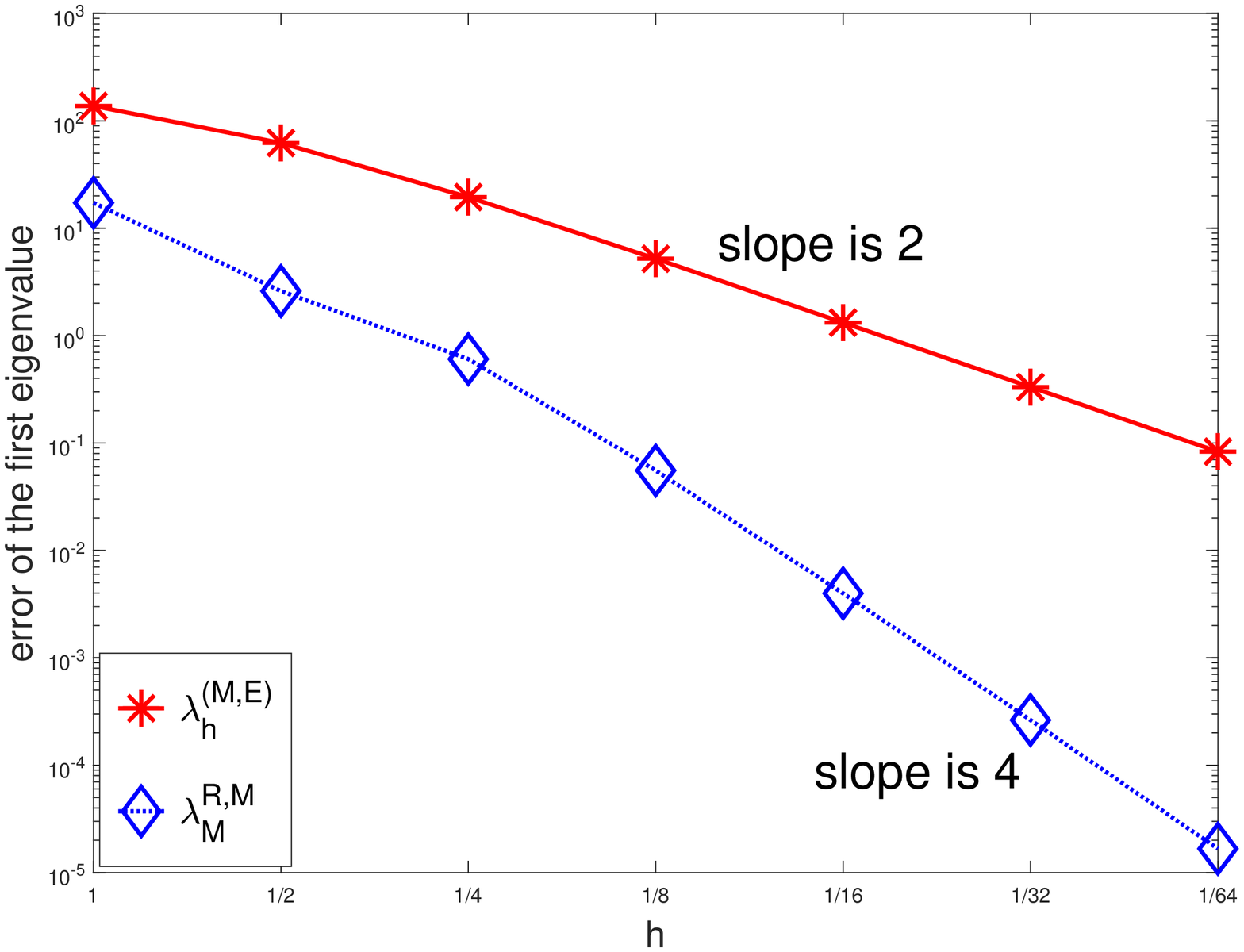}
\caption{\footnotesize{The errors of recovering eigenvalues $\lambda_{\rm M}^{\rm R,\ M}$ and approximate eigenvalues $\lambda_h^{\rm (M, E)}$ by the Morley element for Example 5.}}
\label{fig:squareMorley}
\end{figure}

The problem is solved by the Morley element on the same sequence of uniform triangulations in Example 1. It is known that the first eigenvalue of this problem is $\lambda =4\pi^4$, and the convergence rates of approximate eigenvalues by the Morley element are 2. Figure \ref{fig:squareMorley} reveals that the recovering eigenvalues $\lambda_{\rm M}^{\rm R,\ M}$ converge at a higher rate 4, which is in accordance with Remark \ref{remark:morley}.

\bibliographystyle{plain}
\bibliography{bibifile}

\begin{thebibliography}{10}

\bibitem{bank2003asymptotically}
Randolph~E Bank and Jinchao Xu.
\newblock Asymptotically exact a posteriori error estimators, part i: Grids
  with superconvergence.
\newblock {\em SIAM Journal on Numerical Analysis}, 41(6):2294--2312, 2003.

\bibitem{Blum1990Finite}
H~Blum and R~Rannacher.
\newblock Finite element eigenvalue computation on domains with reentrant
  corners using \protect{Richardson} extrapolation.
\newblock {\em Journal of Computational Mathematics}, 8(4):321--332, 1990.

\bibitem{Brandts1994Superconvergence}
Jan~H Brandts.
\newblock Superconvergence and a posteriori error estimation for triangular
  mixed finite elements.
\newblock {\em Numerische Mathematik}, 68(3):311--324, 1994.

\bibitem{Jan2000Superconvergence}
Jan~H Brandts.
\newblock Superconvergence for triangular order k=1 \protect{Raviart-Thomas}
  mixed finite elements and for triangular standard quadratic finite element
  methods.
\newblock {\em Applied Numerical Mathematics}, 34(1):39--58, 2000.

\bibitem{Brezzi2009On}
Franco Brezzi.
\newblock On the existence, uniqueness and approximation of saddle-point
  problems arising from \protect{Lagrangian} multipliers.
\newblock {\em Revue fran{\c{c}}aise d'automatique, informatique, recherche
  op{\'e}rationnelle. Analyse num{\'e}rique}.

\bibitem{Chen1995High}
Chuanmiao Chen and Yunqing Huang.
\newblock High accuracy theory of finite element method.
\newblock 1995.

\bibitem{Chen2013Superconvergence}
Hongsen Chen and Bo~Li.
\newblock Superconvergence analysis and error expansion for the
  \protect{Wilson} nonconforming finite element.
\newblock {\em Numerische Mathematik}, 69(2):125--140, 2013.

\bibitem{Chen2007Asymptotic}
Wei Chen and Qun Lin.
\newblock Asymptotic expansion and extrapolation for the eigenvalue
  approximation of the biharmonic eigenvalue problem by
  \protect{Ciarlet-Raviart} scheme.
\newblock {\em Advances in Computational Mathematics}, 27(1):95--106, 2007.

\bibitem{Crouzeix1973Conforming}
Michel Crouzeix and P-A Raviart.
\newblock Conforming and nonconforming finite element methods for solving the
  stationary stokes equations.
\newblock {\em Revue fran{\c{c}}aise d'automatique informatique recherche
  op{\'e}rationnelle. Math{\'e}matique}, 7(R3):33--75, 1973.

\bibitem{Ding1990quadrature}
Yanheng Ding and Qun Lin.
\newblock Quadrature and extrapolation for the variable coefficient elliptic
  eigenvalue problem.
\newblock {\em Systems Science Mathematical Sciences}, 3(4):327--336, 1990.

\bibitem{Douglas1985Global}
Jim Douglas and Jean~E Roberts.
\newblock Global estimates for mixed methods for second order elliptic
  problems.
\newblock {\em Mathematics of Computation}, 44(169):39--52, 1985.

\bibitem{Douglas1989Superconvergence}
Jim Douglas and Junping Wang.
\newblock Superconvergence of mixed finite element methods on rectangular
  domains.
\newblock {\em Calcolo}, 26(2-4):121--133, 1989.

\bibitem{Dur1990Superconvergence}
Ricardo Dur$\rm\acute{a}$n.
\newblock Superconvergence for rectangular mixed finite elements.
\newblock {\em Numerische Mathematik}, 58(1):287--298, 1990.

\bibitem{Hu2014Lower}
Jun Hu, Yunqing Huang, and Qun Lin.
\newblock Lower bounds for eigenvalues of elliptic operators: By nonconforming
  finite element methods.
\newblock {\em Journal of Scientific Computing}, 61(1):196--221, 2014.

\bibitem{Hu2015Constructing}
Jun Hu, Yunqing Huang, and Quan Shen.
\newblock Constructing both lower and upper bounds for the eigenvalues of
  elliptic operators by nonconforming finite element methods.
\newblock {\em Numerische Mathematik}, 131(2):273--302, 2015.

\bibitem{Hu2012A}
Jun Hu, Yunqing Huang, and Qun Shen.
\newblock A high accuracy post-processing algorithm for the eigenvalues of
  elliptic operators.
\newblock {\em Journal of Scientific Computing}, 52(2):426--445, 2012.

\bibitem{Hu2015The}
Jun Hu and Rui Ma.
\newblock The \protect{Enriched Crouzeix-Raviart} elements are equivalent to
  the \protect{Raviart-Thomas} elements.
\newblock {\em Journal of Scientific Computing}, 63(2):410--425, 2015.

\bibitem{Hu2016Superconvergence}
Jun Hu and Rui Ma.
\newblock Superconvergence of both the \protect{Crouzeix-Raviart} and
  \protect{Morley} elements.
\newblock {\em Numerische Mathematik}, 132(3):491--509, 2016.

\bibitem{ZHONG2005CONSTRAINED}
Jun Hu and Zhong-Ci Shi.
\newblock Constrained quadrilateral nonconforming rotated \protect{$Q_1$}
  element.
\newblock {\em Journal of Computational Mathematics(Chinese)}, 23(6):561--586,
  2005.

\bibitem{huang2010superconvergent}
Yunqing Huang and Nianyu Yi.
\newblock The superconvergent cluster recovery method.
\newblock {\em Journal of Scientific Computing}, 44(3):301--322, 2010.

\bibitem{Jia2010Approximation}
Shanghui Jia, Hehu Xie, Xiaobo Yin, and Shaoqin Gao.
\newblock Approximation and eigenvalue extrapolation of biharmonic eigenvalue
  problem by nonconforming finite element methods.
\newblock {\em Numerical Methods for Partial Differential Equations},
  24(2):435--448, 2010.

\bibitem{li2017global}
Yuwen Li.
\newblock Global superconvergence of the lowest order mixed finite element on
  mildly structured meshes.
\newblock {\em arXiv preprint arXiv:1712.08316}, 2017.

\bibitem{Lin2005CAN}
Qun Lin.
\newblock Can we compute laplace eigenvalues well, like computing $\pi$?
\newblock {\em International Journal of Information and Systems Sciences},
  1(2):172--183, 2005.

\bibitem{Lin2008New}
Qun Lin, Hung-Tsai Huang, and Zi-Cai Li.
\newblock New expansions of numerical eigenvalues for $-\bigtriangleup
  u=\lambda \rho u$ by nonconforming elements.
\newblock {\em Mathematics of Computation}, 77(264):2061--2084, 2008.

\bibitem{Lin2009New}
Qun Lin, Hung-Tsai Huang, and Zi-Cai Li.
\newblock New expansions of numerical eigenvalues by \protect{Wilson's}
  element.
\newblock {\em Journal of Computational and Applied Mathematics},
  225(1):213--226, 2009.

\bibitem{Lin2007Finite}
Qun Lin and Jiafu Lin.
\newblock Finite element methods: Accuracy and \protect{Improvement}.
\newblock {\em China Sci. Press, Beijing}, 2006.

\bibitem{Lin1984Asymptotic}
Qun Lin and Tao Lu.
\newblock Asymptotic expansions for finite element eigenvalues and finite
  element.
\newblock {\em Bonn. Math. Schrift}, 158:1--10, 1984.

\bibitem{LIN2005ON}
Qun Lin, Lutz Tobiska, and Aihui Zhou.
\newblock On the superconvergence of nonconforming low order finite
  elementsapplied to the poisson equation.
\newblock {\em Ima Journal of Numerical Analysis}, 25(1), 2005.

\bibitem{lin1999high}
Qun Lin and Dongsheng Wu.
\newblock High-accuracy approximations for eigenvalue problems by the carey
  non-conforming finite element.
\newblock {\em International Journal for Numerical Methods in Biomedical
  Engineering}, 15(1):19--31, 1999.

\bibitem{Lin2009Asymptotic}
Qun Lin and Hehu Xie.
\newblock Asymptotic error expansion and \protect{Richardson} extrapolation of
  eigenvalue approximations for second order elliptic problems by the mixed
  finite element method.
\newblock {\em Applied Numerical Mathematics}, 59(8):1884--1893, 2009.

\bibitem{Lin2011Extrapolation}
Qun Lin, Junming Zhou, and Hongtao Chen.
\newblock Extrapolation of three-dimensional eigenvalue finite element
  approximation.
\newblock {\em Mathematics in Practice Theory}, 11(11):132--139, 2011.

\bibitem{J1972Non}
JL~Lions and E~Magenes.
\newblock Non-homogeneous boundary value problems and applications. vol. i.
  translated from the french by p. kenneth.
\newblock {\em Lithos}, 118(3-4):349--364, 1972.

\bibitem{Luo2002High}
Ping Luo and Qun Lin.
\newblock High accuracy analysis of the \protect{Adini's} nonconforming
  element.
\newblock {\em Computing}, 68(1):65--79, 2002.

\bibitem{mao2009high}
Shipeng Mao and Zhong-ci Shi.
\newblock High accuracy analysis of two nonconforming plate elements.
\newblock {\em Numerische Mathematik}, 111(3):407--443, 2009.

\bibitem{Marini1985An}
Luisa~Donatella Marini.
\newblock An inexpensive method for the evaluation of the solution of the
  lowest order \protect{Raviart-Thomas} mixed method.
\newblock {\em Siam Journal on Numerical Analysis}, 22(3):493--496, 1985.

\bibitem{MING2006SUPERCONVERGENCE}
Pingbin Ming, Zhong-ci Shi, and Yun Xu.
\newblock Superconvergence studies of quadrilateral nonconforming rotated
  \protect{$Q_1$} elements.
\newblock {\em International Journal of Numerical Analysis Modeling},
  3(3):322--332, 2006.

\bibitem{Zhang2006Enhancing}
Ahmed Naga, Zhimin Zhang, and Aihui Zhou.
\newblock Enhancing eigenvalue approximation by gradient recovery.
\newblock {\em Journal of Scientific Computing}, (28):1289--1300, 2006.

\bibitem{Raviart1977A}
Pierre-Arnaud Raviart and Jean-Marie Thomas.
\newblock A mixed finite element method for second order elliptic problems.
\newblock {\em Springer Berlin Heidelberg}, (606):292--315, 1977.

\bibitem{ShiWangBook}
Zhong-Ci Shi and Ming Wang.
\newblock The finite element method(\protect{In Chinese}).
\newblock {\em Science Press, Beijing}, 2010.

\bibitem{Zhang2005A}
Zhimin Zhang.
\newblock A posteriori error estimates based on the polynomial preserving
  recovery.
\newblock {\em Siam Journal on Numerical Analysis}, 42(4):1780--1800, 2005.

\bibitem{Zienkiewicz1992The}
Olgierd~Cecil Zienkiewicz and Jian~Zhong Zhu.
\newblock The superconvergent patch recovery and a posteriori error estimates.
  part 2: Error estimates and adaptivity.
\newblock {\em International Journal for Numerical Methods in Engineering},
  33(7):1331--1364, 1992.

\end{thebibliography}

\end{document}